\documentclass[preprint,a4paper,11pt]{elsarticle}

\def\re{\mathbb{R}}

\usepackage{graphicx}
\usepackage{amsmath}
\usepackage{amscd}
\usepackage{amssymb}
\usepackage{graphicx}
\usepackage{subfig}
\usepackage{multirow}
\usepackage{array}

\usepackage{amssymb}

\newtheorem{thm}{Theorem}

\newtheorem{algorithm}{Algorithm}
\newtheorem{subroutine}{Subroutine}
\newtheorem{prop}{Proposition}
\newtheorem{lemma}{Lemma}

\newenvironment{proof}[1][Proof]{\begin{trivlist}
\item[\hskip \labelsep {\bfseries #1}]}{\end{trivlist}}

\begin{document}

\begin{frontmatter}



\title{A Survey of Algorithms for Separable Convex Optimization with Linear Ascending 
Constraints\tnoteref{t1}}
\tnotetext[t1]{This work was supported by the Indo-French Centre for Promotion of Advanced Research 
under project number 5100-IT1. The funding agency provided an unrestricted research grant. It was 
neither involved in the study design, nor in the collection, analysis and interpretation of 
data, nor in the writing of the report, nor in the decision to submit the article for publication.}

\author[ece2]{P T Akhil\corref{cor1}}
\ead{akhilpt@ece.iisc.ernet.in}
\cortext[cor1]{Corresponding author}

\author[ece1]{Rajesh Sundaresan}


\address[ece2]{Department of Electrical Communication Engineering, Indian Institute of Science, 
Bangalore 560012, India.}
 \address[ece1]{Department of Electrical Communication Engineering and Robert Bosch Centre for 
Cyber Physical Systems, Indian Institute of Science, Bangalore 560012, India.}

\begin{abstract}
The paper considers the minimization of  a separable convex function  subject to linear ascending 
constraints. The problem arises as the core optimization in several resource allocation scenarios, 
and is a special case of an optimization of a separable convex function over the bases of a 
polymatroid with a certain structure. The paper presents a survey of state-of-the-art algorithms 
that solve this optimization problem. The algorithms are applicable to the class of separable 
convex objective functions that need not be smooth or strictly convex. When the objective function 
is a so-called $d\operatorname{-}$separable function, a simpler linear time algorithm solves the 
problem. 
\end{abstract}


\end{frontmatter}
\pagestyle{myheadings}
\thispagestyle{plain}
\markboth{P. T. AKHIL AND R. SUNDARESAN}{CONVEX OPTIMIZATION WITH ASCENDING CONSTRAINTS}

\section{Introduction}
\label{sec:intro}
In this paper we consider the following separable convex optimization problem with linear inequality 
constraints. The problem arises in a wide variety of resource allocation settings and we highlight 
several immediately after stating the abstract problem. Let $x=(x(1),x(2),\cdots,x(n)) \in 
\mathbb{R}^{n}$. Let $w_{e}:[0,b_{e}) \rightarrow \overline{\mathbb{R}}$, $e=1,2,\cdots,n$ be  
convex  functions where $0 < b_{e} \leq \infty$ and $\overline{\mathbb{R}} = \mathbb{R} \cup \{ 
-\infty, +\infty \}$ be the extended real line. We wish to minimize a separable objective function 
$W:\mathbb{R}^{n} \rightarrow \overline{\mathbb{R}}$ as in

\begin{align}\label{eqn:minimizeSumSeparableConvexFunctions}
 \text{Problem}~~ \Pi :~~ \text{Minimize}~~~~  & W(x) := \sum_{e=1}^{n}w_{e}(x(e)) \\
  \label{eqn:positivityBounded}
  \text{subject to} ~~~    & x(e)  \in  [0, \beta(e)], ~~~~~~~~e=1,2,\cdots,n,\\
   \label{eqn:ladderConstraint}
                            & \sum_{e=1}^{l} x(e)  \geq  \sum_{e=1}^{l} \alpha(e),
 ~~l=1,2,\cdots,n-1, \\
   \label{eqn:equalityConstraint}
                            &  \sum_{e=1}^{n} x(e) =  \sum_{e=1}^{n} \alpha(e).
\end{align}
We assume $\beta(e) \in (0,b_{e}]$ for $e = 1,2,\cdots,n$, and $\alpha(e) \geq 0$ for 
$e=1,2,\cdots,n$. The inequalities in (\ref{eqn:positivityBounded}) impose positivity and upper 
bound constraints. The inequalities in (\ref{eqn:ladderConstraint}) impose a sequence of {\it 
ascending constraints} with increasing heights $\sum_{e=1}^l \alpha(e)$ indexed by $l$. We also 
assume
\begin{equation}
   \label{eqn:commonSense1}
   \sum_{e=1}^l \alpha(e) \leq \sum_{e=1}^l \beta(e), ~~~~~~~l=1,2,\cdots,n,
 \end{equation}
a necessary and sufficient condition for the feasible set to be nonempty (Lemma \ref{lemma:fb} of 
the appendix).

Problem $\Pi$ arises in the following applications in operations research.
\begin{itemize}
 \item An inventory problem with downward substitution (Wang \cite{201410OL_Ziz}): A firm produces a 
product with $n$ different grades. A higher grade of the product can be substituted for a lower 
grade. The firm has to make an inventory decision on the number of grade $i$ product to stock before 
the demand is known. The objective is to minimize the loss incurred due to mismatch between the 
demand and the supply of products of each grade.
\item The Clark-Scarf series multiechelon inventory model (Clark and Scarf \cite{1960xxMS_ClaSca}): 
The inventory model consists of $n$ facilities arranged in series. The demand is a random variable 
with known distribution of finite variance. The demand is first served at facility $n$ with the 
excess demand at facility $n$ passed on to facility $n-1$, and so on. There are costs involved in 
passing the demand to the next facility. There are also storage costs at the facilities. The problem 
is to find the amount to stock at each of the $n$ facilities. The objective is to minimize the cost 
incurred due to the failures in meeting the demands at the facilities and the storage costs.
\end{itemize}
In addition to the above mentioned cases in operations research, problem $\Pi$ arises in certain 
resource allocation problems in wireless communications where several mobiles simultaneously access 
a common medium. A high-level description is as follows. Each mobile transmits with a certain power 
(measured in joules per second) on a one-dimensional subspace of the available signaling vector 
space. The dimension of the signaling vector space is fewer than the number of mobiles, and 
orthogonalization of mobile transmissions is not possible. If two mobiles' signaling directions are 
not orthogonal, they will interfere with each other and affect each other's transmissions. Problem 
$\Pi$ arises in the optimal allocation of directions (one-dimensional subspaces) to mobiles in each 
of the following settings.
\begin{itemize}
  \item Mobiles have rate requirements (in bits per second) that must be met and the net transmitted 
power is to be minimized (Padakandla and Sundaresan \cite{200910TCOM_PadSun}).
  \item Mobiles have quality of service requirements (in terms of minimum signal to interference and 
noise ratios) that must be met and again the net transmitted power is to be minimized (Viswanath and 
Anantharam \cite[Sec. III]{200206TIT_VisAna}).
  \item Mobiles have maximum transmission power constraints and the total rate achieved across all 
mobiles is to be maximized (Viswanath and Anantharam \cite[Sec. II]{200206TIT_VisAna}).
\end{itemize}
Problem $\Pi$ also arises in the optimization of wireless multiple input multiple output (MIMO) 
systems as follows.
\begin{itemize}
\item Minimize power to meet mean-squared-error quality of service constraints on each of the 
datastreams in a point to point MIMO communication system (Lagunas et al. 
\cite{200405TSP_PalLagCio}).
\item Minimize power in the context of linear transceiver design on MIMO networks with a single 
non-regenerative relay (Sanguinetti and D'Amico \cite{201205TSP_SanDam}). 
\end{itemize}
Problem $\Pi$ also arises in an inequality constrained maximum likelihood estimation problem where 
the parameters of a multinomial distribution are to be estimated subject to the constraint that the 
associated distribution is stochastically smaller than a given distribution (Frank et al. 
\cite{1966xx_HanBruFraHog}).

paper is to go beyond $\mathcal{C}^1$ functions.

The wide range of applications indicated above motivate us to study the abstracted problem $\Pi$ in 
some detail.

The special case $\alpha(1)=\alpha(2)=\cdots=\alpha(n-1)=0$ makes the constraint in 
(\ref{eqn:ladderConstraint}) irrelevant, and problem $\Pi$ reduces to the well-studied minimization 
of a separable convex cost in (\ref{eqn:minimizeSumSeparableConvexFunctions}) subject to boundedness 
and sum constraints, i.e., (\ref{eqn:positivityBounded}) and (\ref{eqn:equalityConstraint}), 
respectively. See Patriksson \cite{200802AEJOR_Pat} for a bibliographical survey of such problems 
with historical remarks and comments on solution approaches.

Morton et al. \cite{1985xxMP_MorRanRin} studied the special case of problem $\Pi$ when $w_{e}(t) = 
\lambda(e)t^{p},~ e=1,2,\cdots,n$, where $p>1$. They characterized the constraint set as the {\it 
bases} of a {\it polymatroid}; we will define these terms soon. The algorithms to minimize a 
separable convex function over the bases of a polymatroid fall into two broad categories: greedy 
algorithms and decomposition algorithms. In the greedy category, the algorithm due to Federgruen 
and 
Groenevelt \cite{1986xxOR_FedGro} has a complexity $\mathcal{O}(B\cdot(\log\,n + F))$, where $B$ is 
the total number of units to be allocated among the $n$ variables and $F$ is the number of steps 
required to check the feasibility at each step of the greedy update. Hochbaum \cite{1994xxMOR_Hoc} 
proposed a variant of the greedy algorithm that uses a scaling technique to reduce the complexity 
to 
$\mathcal{O}\left(n\cdot(\log\,n+F)\cdot\log\left(B/(n\epsilon)\right)\right)$. Hochbaum 
\cite{1994xxMOR_Hoc} points out that $F$ takes $\mathcal{O}(1)$ time for the case of linear 
ascending constraints. 

The algorithms in the decomposition category use a divide and conquer approach. This class of 
algorithms divide the optimization problem into several subproblems which are easier to solve. 
Groenevelt \cite{1991xxEJOR_Gro} provided two algorithms of this variety. Fujishige in 
\cite[Sec.~8]{2003xxELS_Fuj} extended Groenevelt's algorithms to minimize over the bases of a more 
general ``submodular system''. Fujishige's decomposition algorithm requires an oracle for a 
particular step of the algorithm. The oracle gives a base satisfying certain conditions. This oracle 
depends on the particularity of the problem. For example, Groenevelt's decomposition algorithm 
\cite{1991xxEJOR_Gro} is for symmetric polymatroids. Morton et al. \cite{1985xxMP_MorRanRin} 
exploited the even more special structure in the polymatroid arising from the constraints in 
(\ref{eqn:positivityBounded})-(\ref{eqn:equalityConstraint}) and identified an explicit algorithm 
that obviated the need for an oracle. Their procedure however was for the special case of $w_{e}(t) 
= \lambda(e)t^{p}, ~p>1$, as indicated at the beginning of this paragraph. Padakandla and Sundaresan 
\cite{200908SIAM_PadSun} provided an extension of this algorithm to strictly convex 
$\mathcal{C}^{1}$ functions that satisfy certain slope conditions. Their proof of optimality is via 
a verification of Karush-Kuhn-Tucker conditions. Akhil, Singh and Sundaresan 
\cite{2014xxNCC_AkhSinSun} simplified the algorithm of \cite{200908SIAM_PadSun} and relaxed some of 
the constraints imposed by Padakandla and Sundaresan \cite{200908SIAM_PadSun} on the objective 
function. See D'Amico et al. \cite{201411_DamSanPal} for a similar extension of the algorithm of 
\cite{200908SIAM_PadSun}. The complexity of the algorithms of Padakandla and Sundaresan 
\cite{200908SIAM_PadSun} and Akhil, Singh, and Sundaresan \cite{2014xxNCC_AkhSinSun} are at least 
$\mathcal{O}(n^2)$ because a certain nonlinear equation is solved $n^2$ times. Zizhuo Wang's 
algorithm \cite{201410OL_Ziz} reduces the number of times the non-linear equation is solved by a 
factor of $n$. Vidal et al. \cite{201404_VidJaiMac} proposed a decomposition algorithm that solves 
problem $\Pi$ in same number of steps as Hochbaum's greedy algorithm \cite{1994xxMOR_Hoc}.
$\mathcal{O}\left(n\cdot\log\,n\cdot\log\left(B/(n\epsilon)\right)\right)$. 

In this paper, we provide a brief description of the greedy algorithm of Hochbaum 
\cite{1994xxMOR_Hoc} and the decomposition algorithm of Vidal et al \cite{201404_VidJaiMac}. We also 
extend the algorithm of \cite{2014xxNCC_AkhSinSun} to a wider class of separable convex functions, 
such as negatives of piece-wise linear concave utility functions which commonly arise in the 
economics literature. The extended algorithm proposed in this paper works for any convex $w_e$, in 
particular, they need not be strictly convex or differentiable. The proof of correctness of the 
algorithm employs the theory of polymatroids, which is summarized in the next section. 

The complexity of our algorithm for the more general $w_e$ is same as that of Padakandla and 
Sundaresan \cite{200908SIAM_PadSun}, and indeed, our algorithm reduces to that of 
\cite{200908SIAM_PadSun} when $w_e$ are $\mathcal{C}^1$ functions. We therefore refer the reader to 
that paper \cite[Sec.~1]{200908SIAM_PadSun} for a discussion on the relationship of the algorithm to 
those surveyed by Patriksson \cite{200802AEJOR_Pat}, and for performance comparisons with a standard 
optimization tool (\cite[Sec.~4]{200908SIAM_PadSun}).

The decomposition approach leads to an efficient algorithm for a special case of problem $\Pi$ where 
the objective function has the following form:
\begin{equation}\label{eqn:speobj}
W(x)=\sum_{e=1}^{n}d_e\, \phi\left(\frac{x(e)}{d_e}\right).
\end{equation}
The minimizer of (\ref{eqn:speobj}) over the bases of a polymatroid is known to be the {\it 
lexicographically optimal} base with respect to the weight  vector $d=(d_1,d_2,\cdots,d_n)$ (a 
notion introduced by Fujishige in \cite{198005MOR_Fuj}). This lexicographically optimal base 
optimizes $W$ in (\ref{eqn:speobj}) arising from $\phi$ that is strictly convex and continuously 
differentiable. Hence it suffices to consider (\ref{eqn:speobj}) for the case of a quadratic 
$\phi$. 
Veinott Jr. \cite{1971xxMS_Vei} proposed an elegant geometrical characterization of this optimal 
base when the polymatroid is defined by linear ascending constraints. The geometrical 
characterization is that of a {\it taut-string solution} for the minimizer of the optimization 
problem from the graph of the least concave majorant of a set of $n$ points in $x\operatorname{-}y$ 
plane. Though Veinott Jr.'s computation of the least concave majorant of $n$ points requires 
$\mathcal{O}(n^2)$ steps, the string algorithm elucidated in Muckstadt and Sapra \cite[Sec. 
3.2.3]{2010xxSPR_MucSap} in the context of an inventory management problem, finds it in 
$\mathcal{O}(n)$ steps. We briefly discuss the {\it taut-string solution} \cite{1971xxMS_Vei} and 
the string algorithm of Muckstadt and Sapra \cite[Sec. 3.2.3]{2010xxSPR_MucSap}.

The rest of the paper is organized as follows. In the next section, we discuss the preliminaries 
related to polymatroids. In section \ref{sec:results}, we summarize the algorithm that solves 
problem $\Pi$ and state the main results. 
In sections \ref{sec:greedy} and \ref{sec:decompose}, we summarize Hochbaum's and Vidal's algorithms 
respectively. In section \ref{sec:string}, we discuss the {\it taut-string method}. Finally, in 
Section \ref{sec:fba}, we summarize the string algorithm that finds the least concave majorant of a 
set of points.

While our paper is largely a survey, it also contains some novel contributions. The extension to 
general $w_e$ (not necessarily strictly convex and/or not necessarily continuously differentiable 
everywhere in the domain), contained in section \ref{sec:results} and the Appendix, is new. The 
recognition that the string algorithm of Muckstadt and Sapra \cite[Sec. 3.2.3]{2010xxSPR_MucSap} is 
applicable to the problem of Veinott Jr. \cite{1971xxMS_Vei} is also novel.

\section{Preliminaries}
\label{sec:prelim}

In this section, we describe some preliminary results that reduce  problem $\Pi$ to an optimization 
over the  bases of an appropriate polymatroid. The reduction is due to Morton et al. 
\cite{1985xxMP_MorRanRin} and is given here for completeness. We then state a result due to 
Groenevelt \cite{1991xxEJOR_Gro} for polymatroids which was subsequently generalized to submodular 
functions by Fujishige \cite[Sec. 8]{2003xxELS_Fuj}. Groenevelt's result will provide a necessary 
and sufficient condition for optimality. The next section provides an algorithm to arrive at a base 
that satisfies the necessary and sufficient condition of Groenevelt \cite{1991xxEJOR_Gro}. We begin 
with some definitions.

Let $E= \{1,2,\ldots,n\}$. Let $f:2^{E}\rightarrow\mathbb{R}_{+}$ be a \emph{rank} function, i.e., a 
nonnegative real function on the set of subsets of $E$ satisfying
\begin{eqnarray}
f(\emptyset) & = & 0, \\
f(A) & \leq & f(B) ~~~~~~~~~~~~~~~~~~~~~~~~(A \subseteq B \subseteq E), \\
f(A) + f(B) & \geq & f(A \cup B) + f(A \cap B)~~~(A, B \subseteq E).
\end{eqnarray}
The pair $(E,f) $ is called a \emph{polymatroid} with ground set $E$.
For an $x \in \mathbb{R}_{+}^{E}$ and $A \subseteq E$ define
\[
x(A) := \sum_{e \in A} x(e).
\]
A vector $x \in \mathbb{R}_{+}^{E}$ is called an \emph{independent} vector if $x(A) \leq f(A)$ for 
every $A \subseteq E$. Let $P(f)$, the \emph{polymatroidal polyhedron}, denote the set of all 
independent vectors of $(E,f)$. The \emph{base} of the polymatroid $(E,f)$, denoted $B(f)$, is 
defined as
\[
  B(f) := \{ x \in P(f): x(E) = f(E)\}.
\]
These are the maximal elements of $P(f)$ with respect to the partial order ``$\leq$'' on the set 
$\mathbb{R}_{+}^{E}$ defined by component-wise domination, i.e., $x \leq y~  \text{if and only if}~ 
x(e) \leq y(e) ~\text{for every}~ e \in E$. We shall also refer to an element of the base of a 
polymatroid as a base.

For two real sequences $a = (a_1, a_2, \cdots, a_k)$ and $b = (b_1, b_2, \cdots, b_k)$ of same 
length $k$, $a$ is lexicographically greater than or equal to $b$ if for some $j \in \{ 1, 2, 
\cdots, k \}$ we have
\[
  a_i = b_i ~ (i = 1, 2, \cdots, j-1) ~~~~ \mbox{ and } ~~~~ a_j > b_j
\]
or
\[
  a_i = b_i ~ (i = 1, 2, \cdots, k).~~~~~~~~~~~~~~~~~~~~~~~~~~~~
\]
Let $x \in \re_+^{E}$ and let $T(x)$ be the $|E|$-length sequence arranged in the increasing order 
of magnitude. Let $h_e: \re_+ \rightarrow \re$ be a family of continuous and strictly increasing 
functions indexed by $e \in E$. Let $h: \re_+^E \rightarrow \re^E$ be the vector function defined 
as 
$h(x) := (h_e(x(e)), ~e \in E)$. A base $x$ of $(E, f)$ is an $h$-lexicographically optimal base if 
the $|E|$-tuple $T(h(x))$ is lexicographically maximum among all $|E|$-tuples $T(h(y))$ for all 
bases $y$ of $(E, f)$. Let $d \in \re_+^{E}$ with all components strictly positive. For the case 
$h_e=x(e)/d_e$, $h$-lexicographically optimal base is also known as the lexicographically optimal 
base with respect to the weight vector $d$.

For any $e \in E$, define
\[
 \textsf{dep}(x,e,f) = \cap{\{ A ~ | ~ e \in A \subset E, x(A)=f(A)\}},
\]
which in words is the smallest subset among those subsets $A$ of $E$ containing $e$ for which $x(A)$ 
equals the upper bound $f(A)$. Fujishige shows that $\textsf{dep}(x,e,f)-\{e\}$ is $\emptyset$ if 
$e$ and $x \in B(f)$ are such that $x(e)$ cannot be increased without making $x$ dependent. 
Otherwise, $\textsf{dep}(x,e,f)$ is made up of all those  $u \in E$ from which a small amount of 
mass can be moved from $x(u)$ to $x(e)$ yet keeping the new vector independent. Thus $(u,e)$ is 
called an {\it exchangeable pair} if $u \in \textsf{dep}(x,e,f)-\{e\}$.

For a $\beta \in \mathbb{R}_{+}^{E}$, define the set function
\[ f_{\beta}(A) = \min_{D \subseteq A}{\{ f(D) + \beta (A-D)\}} ~~ (A \subset E).
\]
We now state without proof an interesting property of the subset of independent vectors of a 
polymatroid that are dominated by $\beta$. See Fujishige \cite{2003xxELS_Fuj} for a proof.

\vspace*{.1in}
\begin{prop}
\label{propfuj}
The set function $f_{\beta}$ is a rank function and $(E,f_{\beta})$ is a polymatroid. Furthermore, 
$P(f_{\beta})$ is given by $P(f_{\beta}) = \{ x \in P(f): x \leq \beta \}$.
\end{prop}
\vspace*{.1in}

We are now ready to relate the constraint set in problem $\Pi$ to the base of a polymatroid, as done 
by Morton et al \cite{1985xxMP_MorRanRin}. Define
\begin{eqnarray}
  c(0) & := & 0, \nonumber \\
  c(j) & := & \sum_{e=1}^{j} \alpha (e) ~~~~(1 \leq j \leq n),
\end{eqnarray}
and further define
\begin{eqnarray}
  \zeta(A) & := & \max{\{c(j):\{1,2,\ldots,j\}} \subseteq A, 0 \leq j \leq n\}~~~(A \subseteq E), \\
  \label{eqn:f}
  f(A) & := & \zeta (E) - \zeta (E-A) = c(n)- \zeta(E-A)~~~(A \subseteq E).
\end{eqnarray}

\vspace*{.1in}
\begin{prop} \label{propmorton}The following statements hold.
\begin{itemize}
\item The function $f$ in (\ref{eqn:f}) is a rank function and therefore $(E,f)$ is a polymatroid.
\item The set of  $x \in \mathbb{R}_{+}^{E}$ that satisfy the ascending constraints 
(\ref{eqn:ladderConstraint})-(\ref{eqn:equalityConstraint}) equals the base $B(f)$  of the 
polymatroid $(E,f)$.
\item The set of $x \in \mathbb{R_{+}^{E}}$ that satisfy the ascending constraints 
(\ref{eqn:ladderConstraint})-(\ref{eqn:equalityConstraint}) and the domination constraint 
(\ref{eqn:positivityBounded}) equals the base $B(f_{\beta})$ of the polymatroid $(E,f_{\beta})$.
\end{itemize}
\end{prop}
\vspace*{.1in}

See Morton et al. \cite{1985xxMP_MorRanRin} for a proof. Incidentally, this is shown by recognizing 
that $(E,\zeta)$ is a related object called the contrapolymatroid, that the set of all vectors 
meeting the constraints above is the base of the contrapolymatroid, and that the base of the 
contrapolymatroid $(E,\zeta)$ and the base of the polymatroid $(E,f)$ coincide. The above 
proposition thus says that problem $\Pi$ is simply a special case of $\Pi_{1}$ below with 
$g=f_{\beta}$.

Let $(E,g)$ be a polymatroid. Our interest is in the following.
\begin{eqnarray}
  \label{eqn:separableObjective}
  \text{Problem}~~\Pi_{1}:~~~~ \mbox{Minimize } & \sum_{e \in E} w_e(x(e)) \\
  \nonumber
  \mbox{subject to } & x \in B(g).
\end{eqnarray}

We next state a necessary and sufficient condition that an optimal base satisfies. For each $e \in 
E$, define $w_{e}^{+}$ and $w_{e}^{-}$ to be the right and left derivatives of $w_{e}$.

\vspace*{.1in}
\begin{thm} \label{thm:opt}
A base $x \in B(g)$ is an optimal solution to problem $\Pi_{1}$ if and only if for each exchangeable 
pair $(u,e)$ associated with base $x$ (i.e., $u \in \textsf{dep}(x,e,g)-\{e\}$), we have $ 
w_{e}^{+}(x(e)) \geq w_{u}^{-}(x(u))$.
\end{thm}
\vspace*{.1in}

The result is due to Groenevelt \cite{1991xxEJOR_Gro}. See Fujishige \cite[Th. 8.1]{2003xxELS_Fuj} 
for a proof of the more general result on submodular systems. In the next section, we provide an 
algorithm to arrive at a base that satisfies Groenevelt's necessary and sufficient condition.

An alternate characterization of the optimal base when $w_e$ is strictly convex and continuously 
differentiable is the following. Let $h_e$ be the derivative of $w_e$. Note that $h_e$ is continuous 
and strictly increasing. 
\vspace*{.1in}
\begin{thm} \label{thm:opt2}
For each $e \in E$, let $w_e$ be strictly convex and continuously differentiable with $h_e$ as its 
derivative. A base $x \in B(g)$ is an optimal solution to problem $\Pi_{1}$ if and only if $x$ is an 
$h$-lexicographically optimal base.
\end{thm}
\vspace*{.1in}

The result is due to Fujishige \cite{2003xxELS_Fuj}. See Fujishige \cite[Th. 9.1]{2003xxELS_Fuj} for 
the proof.
\section{Some New Results}
\label{sec:results}

Fujishige \cite{2003xxELS_Fuj} provides an algorithm called the \emph{decomposition} algorithm to 
find an increasing chain  of subsets of $E$ as a key step to finding the optimal solution. However, 
in this section, we extend an algorithm of Padakandla and Sundaresan \cite{200908SIAM_PadSun}, which 
is itself an extension of the algorithm of Morton et al \cite{1985xxMP_MorRanRin}, to arrive 
directly at a chain and thence the optimal solution. This algorithm reduces to that of Padakandla 
and Sundaresan \cite{200908SIAM_PadSun} for strictly convex $\mathcal{C}^{1}$ functions  $w_{e}$, $e 
\in E$, and to that of  Morton et al. \cite{1985xxMP_MorRanRin} for the case when $w_{e}(\zeta) = 
\lambda(e) \zeta^{p}$ for each $e \in E$, and $p > 1$.

The following algorithm seeks to identify the desired chain that will yield an $x$ that meets the 
necessary and sufficient condition of Theorem \ref{thm:opt}. Define the generalized inverse 
functions of $w_{e}$ as
\begin{eqnarray*}
  i_{e}^{-}(\eta) & := & \inf{\{\zeta: w_{e}^{+}(\zeta) \geq \eta\}}, \\
  i_{e}^{+}(\eta) & := & \sup{\{\zeta: w_{e}^{-}(\zeta) \leq \eta\}}.
\end{eqnarray*}
See figures \ref{fig:figure1} and \ref{fig:figure2}. The function $w_e^+$ is a right-continuous 
nondecreasing function while its inverse $i_e^-$ is a left-continuous nondecreasing function. On the 
other hand, $w_e^-$ and $i_e^+$ are left-continuous and right-continuous, respectively.

\begin{figure}[ht]
\centering
\includegraphics[scale=0.6]{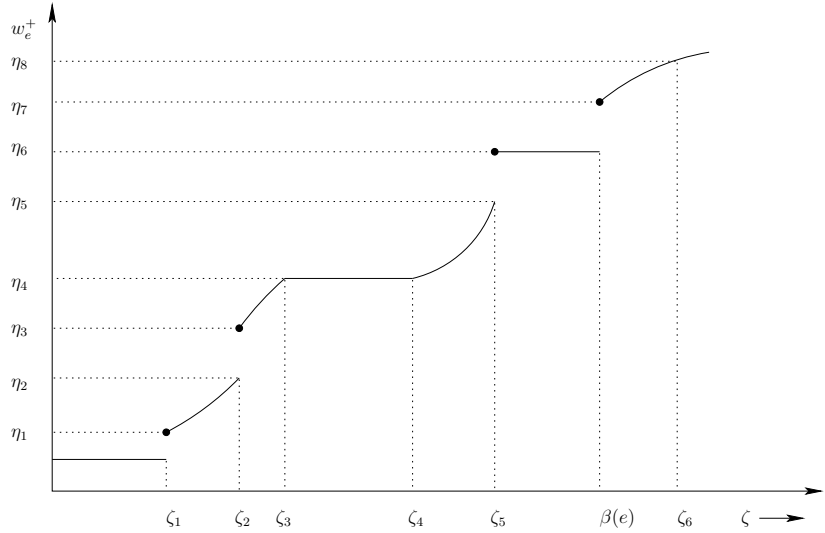}
\caption{The function $w_{e}^{+}(\zeta)$ as a function of $\zeta$.}
\label{fig:figure1}
\end{figure}

\begin{figure}[ht]
\centering
\includegraphics[scale=0.6]{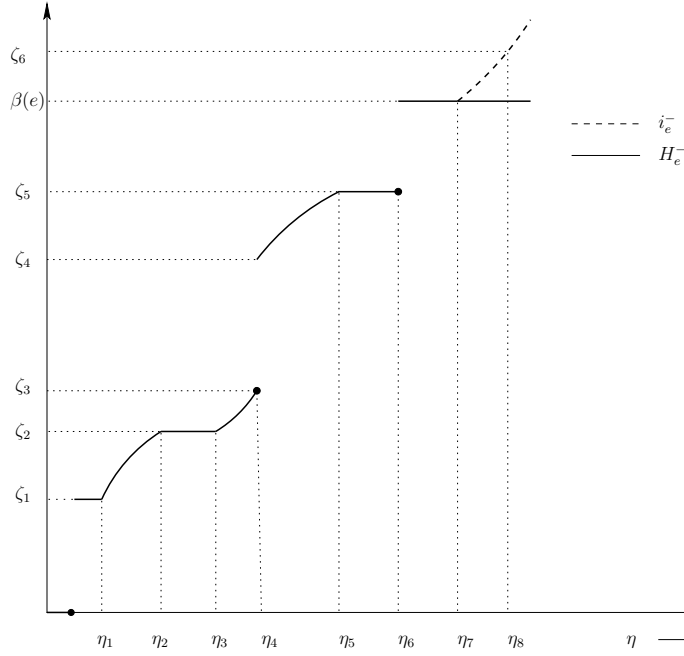}
\caption{The functions $i_{e}^{-}(\eta)$ and $H_{e}^{-}(\eta)$ as a function of $\eta$. They 
coincide for $\eta \leq \eta_7$.}
\label{fig:figure2}
\end{figure}
It is straightforward to see that $i_{e}^{-}(\eta) \leq i_{e}^{+}(\eta)$ with strict inequality only 
for those $\eta$ for which the interval $\{\zeta:w_{e}^{+}(\zeta) = w_{e}^{-}(\zeta) = \eta \}$ is 
not a singleton. For readers unfamiliar with these notions, it may be beneficial to keep the simple 
case when $w_e \in \mathcal{C}^1$ in mind, because in this case $w_e^+ = w_e^- = w'_e$ is the 
derivative of $w_e$, and $i_e^+ = i_e^- = (w'_e)^{-1}$ is the inverse of $w'_e$.

Define their saturated and truncated counterparts as
\begin{displaymath}
  H_{e}^{-}(\eta) = \left\{
     \begin{array}{ll}
       0 & \eta \leq  w_{e}^{+}(0)\\
       i_{e}^{-}(\eta) & w_{e}^{+}(0) < \eta \leq  w_{e}^{-}(\beta(e))\\
        \beta(e) & \eta > w_{e}^{-}(\beta(e)),
     \end{array}
   \right.
\end{displaymath}
and
\begin{displaymath}
  H_{e}^{+}(\eta) = \left\{
     \begin{array}{ll}
       0 & \eta <  w_{e}^{+}(0)\\
       i_{e}^{+}(\eta) & w_{e}^{+}(0) \leq \eta <  w_{e}^{-}(\beta(e))\\
        \beta(e) & \eta \geq w_{e}^{-}(\beta(e)),
     \end{array}
   \right.
\end{displaymath}
respectively. We now state the main algorithm with two associated subroutines.

\vspace*{.15in}
\begin{algorithm}
\label{algorithm}
\begin{itemize}
\item Let $s(0)=n+1$ and set $s(j) \in \{ 1,2,\ldots,n\}$ recursively for $j=1,2,\ldots$ as 
follows.	
\item For each $l$, $1 \leq l < s(j-1)$, pick $\eta_{l}^{j}$ to be the smallest $\eta$ satisfying
\begin{equation} \label{eqn:zero}
 \sum_{e=l}^{s(j-1)-1}H_{e}^-(\eta) \leq \sum_{e=l}^{s(j-1)-1}\alpha(e) \leq 
\sum_{e=l}^{s(j-1)-1}H_{e}^+(\eta)~~~~(1 \leq l < s(j-1)),
\end{equation}
if such an $\eta$ exists. If not, $\eta_l^j$ is undefined. Thus, if $\eta_l^j$ is defined, we have
\begin{equation} \label{eqn:one}
 \sum_{e=l}^{s(j-1)-1}H_{e}^-(\eta_{l}^{j}) \leq \sum_{e=l}^{s(j-1)-1}\alpha(e) \leq 
\sum_{e=l}^{s(j-1)-1}H_{e}^+(\eta_{l}^{j})~~~~(1 \leq l < s(j-1)).
\end{equation}
\item Let
\begin{equation}
\label{eqn:gamma}
\Gamma_{j} = \min\{\eta_{l}^{j} ~|~ 1 \leq l < s(j-1) \mbox{ and } \eta_{l}^{j} \mbox{ exists}\},
\end{equation}
and let $l_{1}, l_{2},\ldots,l_{r}$ be the indices that attain the minimum in~(\ref{eqn:gamma}), 
ordered so that  $s(j-1)>l_{1}>l_{2}>\ldots>l_{r}\geq 1$. The quantity $r$ denotes the number of 
tied indices.
\item Set $s(j)$ as given by the output of Subroutine \ref{subroutine1}.
\item Set values  $x(e)$ for $e = s(j),s(j)+1,\ldots,s(j-1)-1$ as given by the output of Subroutine 
\ref{subroutine2}.
\item If $s(j)=1$, exit. $\hfill \Box$
\end{itemize}
\end{algorithm}
\vspace*{.15in}

We next describe the two subroutines referred to in Algorithm \ref{algorithm}.

The first one sets $s(j)$. One property that the obtained $s(j)$ from Subroutine \ref{subroutine1} 
will have, and we will soon show this, is that constraints (\ref{eqn:ladderConstraint}) are met with 
equality whenever $l$ equals an $s(j)$ for some iterate $j$. We will also see that Subroutine 
\ref{subroutine2} will set a variable $x(e)$ to $H_e^+(\Gamma_j)$ whenever possible. For feasibility 
then, the partial sums of $H_e^+(\Gamma_j)$ from $s(j)$ to each of the tied indices $l_i$ should 
exceed the corresponding partial sums of $\alpha(e)$. So $s(j)$ should be chosen carefully, which is 
the purpose of the next described Subroutine \ref{subroutine1}.

\vspace*{.15in}
\begin{subroutine}
\label{subroutine1}
This subroutine takes as input $s(j-1),r,l_{1},l_{2},\cdots,l_{r}$ and $\Gamma_{j}$. It returns 
$s(j)$.
\begin{itemize}
 \item If there is a unique index $l_{1}$ that achieves the minimum, set $s(j)=l_{1}$. Otherwise we 
have $s(j-1)>l_{1} > l_{2} > \ldots > l_{r}$. Choose $s(j)$ as follows.
Consider the inequality
\begin{equation}
  \label{eqn:subroutine1}
  \sum_{e=m}^{k-1}H_{e}^{+}(\Gamma_{j}) \geq \sum_{e=m}^{k-1}\alpha(e),
\end{equation}
where $k$ and $m$ are iterates.
\begin{enumerate}
\item Initialize $i=1$, $t=r$, $k=l_{i}$, $m=l_{t}$.
\item Repeat \newline
$~~~~~~~$If (\ref{eqn:subroutine1}) holds, $i \leftarrow i+1$, $k \leftarrow l_{i}$.\newline
$~~~~~~~$else $t \leftarrow t-1$, $m \leftarrow l_{t}$ \newline
until $i=t$.
\item Set $s(j)=l_{i}$ and return $s(j)$. $\hfill \Box$
\end{enumerate}
\end{itemize}
\end{subroutine}
\vspace*{.15in}

The second subroutine sets the variables for a subset of $E$.

\vspace*{.15in}
\begin{subroutine}
\label{subroutine2}
This subroutine takes as input the following quantities:
$$s(j), \Gamma_{j}, I=\{l_{1}, l_{2}, \ldots, l_{p}\} \mbox{ with } s(j-1) > l_{1} > l_{2} >\ldots > 
l_{p} > s(j).$$
These $l_{i}$'s are the set of tied indices that satisfy~(\ref{eqn:one}) at the $j^{th}$ iteration, 
and are strictly larger than $s(j)$ set by Subroutine~\ref{subroutine1}. If $I$ is empty, $p$ is 
taken to be zero. It sets $x(e)$ for $e = s(j), s(j)+1, \ldots, s(j-1)-1$.
\begin{enumerate}

  \item Initialize  $~m=0, ~p_{m} = p, ~l_0^{m}= s(j-1), ~l_{p_{m}+1}^{m}=s(j), ~I_{m}=I, 
~l_{i}^{m}=l_{i} ~~(1 \leq i \leq p_{m})$.

  \item \label{step} Calculate
  \begin{equation}\label{eqn:mingamma}
     \gamma = \min{\left\{\sum_{e=l_{i}^{m}}^{l_{0}^{m}-1}\alpha(e)-\sum_{e=l_{i}^{m}}^{l_{0}^{m}-1} 
H_{e}^{-}(\Gamma_{j}),~\forall~ i=1,\ldots , p_{m}+1\right\} }.
  \end{equation}
  \newline Let $t$ be the largest index in (\ref{eqn:mingamma}) such that
  $$\gamma= \sum_{e=l_{t}^{m}}^{l_{0}^{m}-1}\alpha(e)-\sum_{e=l_{t}^{m}}^{l_{0}^{m}-1} 
H_{e}^{-}(\Gamma_{j}).$$

  \item Assign $x(e), e \in \{ l_{0}^{m}-1,l_{0}^{m}-2,\ldots,l_{1}^{m} \}$ such that $x(e) \in 
[H_{e}^{-}(\Gamma_{j}),H_{e}^{+}(\Gamma_{j})]$ and
  \begin{equation}
     \label{assign1}
     \sum_{e=l_{1}^{m}}^{l_{0}^{m}-1}x(e)=\gamma + 
\sum_{e=l_{1}^{m}}^{l_{0}^{m}-1}H_{e}^{-}(\Gamma_{j}).
  \end{equation}
  For $e \in \{ l_{1}^{m}-1,l_{1}^{m}-2,\ldots, l_{t}^{m}\}$, assign
  \begin{equation}
    \label{assign2}
    x(e)=H_{e}^{-}(\Gamma_{j}).
  \end{equation}

  \item If $l_{t}^{m}=s(j)$, exit.

  \item Let $I_{m+1}$ be the indices $l$ in $I \cap \{ l: l<l_{t}^{m} \}$ that satisfy
  \begin{equation}\label{eqn:srt27}
     \sum_{e=l}^{l_{t}^{m}-1}H_{e}^{+}(\Gamma_{j}) \geq \sum_{e=l}^{l_{t}^{m}-1}\alpha(e).
  \end{equation}
  \newline
  Set $p_{m+1}=\lvert I_{m+1} \rvert$. Further, set $l_{0}^{m+1}=l_{t}^{m}$ and 
$l_{p_{m+1}+1}^{m+1}=s(j)$. Denote the indices in $I_{m+1}$ as 
$l_{1}^{m+1},l_{2}^{m+1},\ldots,l_{p_{m+1}}^{m+1}$ such that 
$l_{1}^{m+1}>l_{2}^{m+1}>\cdots > l_{p_{m+1}}^{m+1}$.
  Set $m \leftarrow m+1$. Go to step \ref{step}. $\hfill \Box$
\end{enumerate}
\end{subroutine}
\vspace*{.15in}

We next formally state correctness and optimality properties of the algorithm and prove them in the 
following section. That the operations in the algorithm are all well defined can be gleaned from 
the 
proof given in the appendix.

\vspace*{.15in}
\begin{thm}\label{th3}
If the feasible set is nonempty, Algorithm \ref{algorithm} runs to completion and puts out a 
feasible vector. This output solves problem $\Pi$.
\end{thm}
\vspace*{.15in}

Observe that the hypothesis is the natural minimum requirement that the feasible set is nonempty. 
See Lemma \ref{lemma:fb} for a natural necessary and sufficient condition for a nonempty feasible 
set.

The algorithm for the special case when $w_{e}$ are strictly convex and 
$\mathcal{C}^{1}$ is given in \cite{2014xxNCC_AkhSinSun}. The algorithm is much simpler in this 
case. 
The  Subroutine \ref{subroutine1} chooses $l_{r}$, the smallest index satisfying~(\ref{eqn:one}), 
as $s(j)$ since~(\ref{eqn:subroutine1})  holds true for $m=l_{r}$ and $k=l_{i}$ for all  
$i=1,2,\ldots,r-1$. Since $H_{e}=H_{e}^{-}=H_{e}^{+}$, Subroutine \ref{subroutine2} assigns 
$x(e)=H_{e}(\Gamma_{j})$ for all $e \in \{s(j),s(j)+1,\ldots,s(j-1)-1\}$.

See Appendix for the proof of Theorem \ref{th3}.

\section{A Greedy Algorithm}\label{sec:greedy}
In section \ref{sec:results}, we extended the algorithm of Akhil, Singh, and Sundaresan 
\cite{2014xxNCC_AkhSinSun}, that solves problem $\Pi$, to separable convex functions that are not 
strictly convex or differentiable. In this section, we describe a more efficient algorithm proposed 
by Hochbaum \cite{1994xxMOR_Hoc} (with a correction note by Moriguchi et al. 
\cite{200405MOR_MorShi}) that provides an $\epsilon\operatorname{-}$optimal solution to $\Pi_1$. It 
is based on the greedy approach. 

Consider the following discrete resource allocation problem that is problem $\Pi_1$ with variables 
restricted to integers. 
\begin{align}
  \label{eqn:separableObjectivediscrete}
  \text{Problem}~~\Pi_{2}:~~~~ \mbox{Minimize } &\sum_{e \in E} w_e(x(e)) ~~~~~~~~~~~~~~~~~~~\\
  \nonumber
  \mbox{subject to }  &x \in B(g),\\ \nonumber
                    &x \in \mathbb{Z}^{E}_{+}.
\end{align}
The set of vectors satisfying the constraint set of problem $\Pi_2$ form the bases of the 
polymatroid $(E,g)$ defined over integers. Problem $\Pi_2$ can be solved by the greedy algorithm 
(Federgruen and Groenevelt \cite{1986xxOR_FedGro}). Starting with an initial allocation $x=0$, this 
algorithm increases the value of a variable by one unit if the corresponding decrease in the 
objective function is largest among all possible feasible increments. The complexity of the greedy 
algorithm is $\mathcal{O}\left(B(\log\,n+F)\right)$, where $F$ is the number of operations required 
to check the feasibility of a given increment in a single variable. The $\log\,n$ complexity is to 
keep a sorted array of the marginal decrease in the objective function. Therefore, the complexity of 
the greedy algorithm is exponential in the number of input bits to the algorithm. 

Hochbaum's algorithm, referred to as General Allocation Procedure ($\textsf{GAP}$) in 
\cite{1994xxMOR_Hoc}, combines the greedy algorithm with a scaling technique. Hochbaum considers an 
additional constraint $x \geq l$, where $l=(l_1,l_2,\cdots,l_n)$. The modified greedy algorithm 
consists of a subroutine $\textsf{greedy}(s,\tilde{l})$ that finds the variable that has the largest 
decrease in the objective function value among all variables that can be increased by one unit 
without violating feasibility, and increases it by $s$ units. The subroutine starts with an initial 
allocation $x= \tilde{l}$. Define $\mathbf{e} \in \mathbb{R}^{E}$ as $\mathbf{e} = (1,1,\cdots,1)$ 
and $\mathbf{e}^{k} \in \mathbb{R}^{E}$ as  $\mathbf{e}^{k}(k)=1 ~\mbox{and}~ \mathbf{e}^{k}(j)=0, j 
\neq i$. Let the total number of units to be allocated among $n$ variables be $B$, i.e., $g(E)=B$. 
The subroutine $\textsf{greedy}(s,\tilde{l})$ is described below (including the Moriguchi et al. 
\cite{200405MOR_MorShi} correction).
\vspace*{3mm}
\begin{algorithm}
${\sf greedy}(s,\tilde{l})$
\begin{enumerate}
  \item $x = \tilde{l}, \tilde{B}=B-\tilde{l}\cdot\mathbf{e}, \tilde{E}=\{1,2,\cdots,n\}.$
  \item \label{item2:greedy}$k= \arg\min_{j \in \tilde{E}}\{w_j(x(j)+1)-w_j(x(j))\}$.
  \item {\bf If} $x+\mathbf{e}^{k}$ infeasible
        \begin{itemize}
         \item []$\tilde{E}\leftarrow \tilde{E}\backslash \{k\},\delta_k=s$.
        \end{itemize}
 \item [] {\bf else}, if $x+s\cdot\mathbf{e}^{k}$ infeasible
         \begin{itemize}
          \item [] $\tilde{E} \leftarrow \tilde{E}\backslash \{k\}, \alpha^{\prime}= \hat{c}(x,k), 
x\leftarrow x+\alpha^{\prime}\cdot\mathbf{e}^{k}, \tilde{B}\leftarrow \tilde{B}-\alpha^{\prime}, 
\delta_k=\alpha^{\prime}$.
          \end{itemize}
 \item[] {\bf else} 
    \begin{itemize}
    \item [] $x\leftarrow x+s\cdot\mathbf{e}^{k},\tilde{B}\leftarrow\tilde{B}-s, \delta_k=s.$
    \end{itemize} 
\item {\bf If} $\tilde{B}=0$ or $\tilde{E}= \emptyset$
      \begin{itemize}
       \item [] Output $x$ and STOP. 
      \end{itemize}
\item [] {\bf else}
       \begin{itemize}
         \item [] Go to step \ref{item2:greedy}.
        \end{itemize} 
 \end{enumerate}
\end{algorithm}
\vspace*{3mm}
$\hat{c}(x,k)$ is the saturation capacity defined as the maximum amount by which $x(k)$ can be 
increased without violating feasibility and is given as follows.
\[ \hat{c}(x,k) = \min\{g(A)-\sum_{j\in A}x(j)~\vert~ k\in A \subseteq E\}.\]
The value is then recorded in the variable $\delta_k$.

Hochbaum \cite{1994xxMOR_Hoc}, Moriguchi et al. \cite{200405MOR_MorShi} showed the proximity result 
that if $x^{\star}$ is the optimal solution to $\Pi_2$ and $x^{(s)}$ is the output of 
$\textsf{greedy}(s,\tilde{l})$, then, with $\delta= (\delta_1,\delta_2,\cdots,\delta_n)$,
\begin{equation}\label{eqn:proximity}
 x^{\star} \stackrel{(a)}{>} x^{(s)}-\delta \geq x^{(s)}-s\cdot\mathbf{e}.
\end{equation}
 $\textsf{GAP}$ executes $\textsf{greedy}(s,\tilde{l})$ in each iteration starting  from $s= \left 
\lceil \frac{B}{2n}\right \rceil$ in the first iteration and halving it in each iteration till it 
reaches unity. The output of $\textsf{greedy}(s,\tilde{l})$, deducted by $s$ units, provides an 
increasingly tighter lower bound to $x^\star$ in each iteration. The lowerbound serves as initial 
allocation for the variable $x$ in $\textsf{greedy}(s,\tilde{l})$ in the following iteration. When 
$s=1$, $\textsf{greedy}(s,\tilde{l})$ puts out the optimum value of $\Pi_2$. A formal description of 
$\textsf{GAP}$ is given below.  

\vspace*{3mm}
\begin{algorithm}{\sf GAP}
\begin{enumerate}
  \item $i=0, s_0=\left\lceil \frac{B}{2n}\right\rceil, l^{(s_0)}=l$.
  \item \label{item2:GAP}$x^{(s_i)} =  {\sf greedy}(s_i,l^{(s_i)})$.
  \item {\bf If} $s_i=1$  
       \begin{itemize}
       \item [] Output $x^{\star} = x^{(s_i)}$ and {\bf STOP}.
       \end{itemize}
  \item [] {\bf else} 
       \begin{itemize}
       \item [] $l^{(s_{i+1})}=\max\{l^{(s_i)},x^{(s_i)}-s_i\cdot\mathbf{e}\}, s_{i+1}= \left\lceil 
\frac{s_i}{2}\right\rceil, i\leftarrow i+1$. Go to step \ref{item2:GAP}.
       \end{itemize}
  \end{enumerate}
\end{algorithm}

\vspace*{3mm}
 
As a consequence of the proximity theorem, $l^{(s_i)}$, $i=1,2,\cdots$ form an increasing sequence 
of lower bounds of $x^{\star}$. Hence $\textsf{GAP}$ tightens the lowerbound on $x^{\star}$ in each 
iteration. 

Let $x \in P(g)$ , $\tilde{E}=\{j~\vert ~ x+\mathbf{e}^{j} \mbox{ is feasible)}\}$. If 
\begin{equation} \label{eqn:incrementcondition}
k \in \tilde{E} ~\mbox{ and } ~k=\arg\min_{j \in \tilde{E}}\{w_j(x(j)+1)-w_j(x(j)\},
\end{equation}
then \cite[Cor. 1]{1996xx_GilKovZap} tells that $x^{\star}(k) \geq x(k)$. Clearly, 
$\textsf{greedy}(s,\tilde{l})$ picks exactly such a $k$ as in (\ref{eqn:incrementcondition}) to 
update. Hence deducting the last increment of each variable from the output of 
$\textsf{greedy}(s,\tilde{l})$ results in a lower bound to the optimal value $x^{\star}$ which 
explains the inequality (a) of (\ref{eqn:proximity}).

The complexity of Hochbaum's algorithm is $\mathcal{O}(n\cdot(\log\,n+ F)\cdot \log\frac{B}{n})$. 
Hochbaum \cite{1994xxMOR_Hoc} showed that, for the case of linear ascending constraints, 
feasibility check is equivalent to the disjoint set union problem. A feasibility check step is same 
as a 'union-find' operation. A series of $n \cdot \log\left(\frac{B}{n}\right)$ 'union-find' 
operations can be done in $n \cdot \log\left(\frac{B}{n}\right)$ time \cite{1985xx_GabTar}. Hence 
$F$ is $\mathcal{O}(1)$.

Let $x_c^{\star}$ be the solution to $\Pi_1$ (with the continuous variables). Let $x^{\star}$ be 
the 
output of $\textsf{GAP}$. Moriguchi et al. \cite[Th. 1.3]{2011xx_MorShiTsu} showed that,
\[\|x_c^{\star}-x^{\star}\|_{\infty} \leq n-1.\]
By incrementing in steps of $\epsilon$ instead of $1$, it is now clear that $\textsf{GAP}$ can be 
used to find an $\epsilon\operatorname{-}$optimal solution in time 
$\mathcal{O}\left(n\cdot(\log\,n+F)\cdot\log\left(\frac{B}{n\epsilon}\right)\right)$ 

\section{A decomposition algorithm}
\label{sec:decompose}
In this section, we describe a decomposition algorithm proposed by Vidal et al. 
\cite{201404_VidJaiMac} that solves problem $\Pi$. The complexity of the algorithm in solving $\Pi$ 
is same as that of Hochbaum's algorithm. For the case when problem $\Pi$ has $n$ variables and $m$ 
($m < n$) ascending constraints, the performance of the algorithm beats Hochbaum's algorithm. The 
problem for this case is as follows. Let $s[0]=0,s[m]=n$ and $s[i] \in \{1,2,\cdots,n\}$ with $s[1] 
< s[2] < \cdots < s[m-1] < n$. 
\begin{align}\label{eqn:minimizeSumSeparableConvexFunctions2}
 \text{Problem}~~ \Pi_3 :~~ \text{Minimize}~~~~  & W(x) := \sum_{e=1}^{n}w_{e}(x(e)) \\
  \label{eqn:positivityBounded2}
  \text{subject to} ~~~    & x(e)  \in  [0, \beta(e)]\cap \mathbb{Z}, ~~~~~~~~e=1,2,\cdots,n,\\
   \label{eqn:ladderConstraint2}
                            & \sum_{e=1}^{s[l]} x(e)  \leq  a_l,
 ~~l=1,2,\cdots,m-1, \\
   \label{eqn:equalityConstraint2}
                            &  \sum_{e=1}^{n} x(e) = B. 
\end{align}
The algorithm solves $\Pi_3$ by a recursive process that leads to a hierarchy of sub-problems 
spanning across $1+\lceil\log\,m\rceil$ levels. At each level, multiple sub-problems are solved. 
Each sub-problem involves a subset of variables with optimization done over a single sum constraint 
and upper and lower bounds on variables. The solution to these sub-problems bounds the value of the 
respective variables in the sub-problems in the next higher level. 

We now give the main procedure that involves the tightening of the ascending constraints in $\Pi_3$ 
using the upper bounds on the variables and a call to the main recursive procedure 
$\textsf{Nestedsolve}(1,m)$ which will be described soon.

\begin{algorithm}{\sf General Solution Procedure}
 \begin{itemize}
  \item Tightening:
  \end{itemize}
\begin{enumerate}
  \item $\bar{a}_0 = 0$; $\bar{a}_m = B$.
  \item {\bf for} $i=1$ to $m-1$ {\bf do}
\begin{itemize}
  \item [] $\bar{a}_i \leftarrow 
\min\left\{\bar{a}_{i-1}+\sum_{e=s[i-1]+1}^{s[i]}\beta(e),a_i\right 
\}$.
\end{itemize}
\end{enumerate}  
\begin{itemize}
\item Hierarchical Resolution:
\end{itemize}
\begin{enumerate}  
\setcounter{enumi}{2}
  \item $(\bar{c}_1,\bar{c}_2,\cdots,\bar{c}_n) = (0,0,\cdots,0)$.
  \item $(\bar{d}_1,\bar{d}_2,\cdots,\bar{d}_n) =(\beta(1),\beta(2),\cdots,\beta(n))$.
  \item $(x(1),x(2),\cdots,x(n)) = {\sf Nestedsolve}(1,m)$.
  \item {\bf return} $(x(1),x(2),\cdots,x(n))$.
 \end{enumerate}
\end{algorithm}
The output of the recursive procedure $\textsf{Nestedsolve}(v,w)$ minimizes the following 
optimization problem.
\begin{align}\label{eqn:minimizeSumSeparableConvexFunctions3}
 \text{\textsf{Nested}}(v,w)~~  :~~ \text{Minimize}~~~~  & W(x) := 
\sum_{e=s[v-1]+1}^{s[w]}w_{e}(x(e)) \\
  \label{eqn:positivityBounded3}
  \text{subject to} ~~~    & 0 \leq x(e)  \leq \beta(e), ~~~~~~~~e=s[v-1]+1,\cdots,s[w],\\
   \label{eqn:ladderConstraint3}
                             \sum_{e=s[v-1]+1}^{s[l]}& x(e)  \leq  \bar{a}_l-\bar{a}_{v-1},
 ~~l=v,v+1,\cdots,w-1, \\
   \label{eqn:equalityConstraint3}
                            \sum_{e=s[v-1]+1}^{s[w]} &x(e) =\bar{a}_w-\bar{a}_{v-1}. 
\end{align}
The procedure $\textsf{Nestedsolve}(v,w)$ recursively solves the above problem by solving the 
sub-problems $\textsf{Nestedsolve}(v,t)$ and $\textsf{Nestedsolve}(t+1,w)$ where 
$t=\left\lfloor\frac{v+w}{2} \right\rfloor$. These sub-problems are further divided in the same 
manner and at the lowest level, the sub-problems involve optimization over a single sum constraint. 
Consider the following optimization problem with single sum constraint and bounds on variables.
  \begin{align}\label{eqn:minimizeSumSeparableConvexFunctions4}
 \text{\textsf{RAP}}(v,w)~~  :~~ \text{Minimize}~~~~  & W(x) := \sum_{e=s[v-1]+1}^{s[w]}w_{e}(x(e)) 
\\
       \label{eqn:positivityBounded4}
             \text{subject to}& ~~~~  \hat{c}_e \leq x(e)  \leq \hat{d}_e, 
~~~~~e=s[v-1]+1,\cdots,s[w],  \\
 \label{eqn:equalityConstraint4}       & \sum_{e=s[v-1]+1}^{s[w]} x(e) =\bar{a}_w-\bar{a}_{v-1}.
\end{align}
The subproblems at the lowest level of the recursion is of the form $\textsf{RAP}(v,v)$ with 
$\hat{c}_e=0$ and $\hat{d}_e= \beta_e$ for $e=s[v-1]+1,\cdots,s[v]$. 
We now describe the procedure to obtain the optimal solution to $\textsf{Nested}(v,w)$ from the 
solutions to $\textsf{Nested}(v,t)$ and $\textsf{Nested}(t+1,w)$. Let the solution to 
$\textsf{Nested}(v,t)$ and $\textsf{Nested}(t+1,w)$ be 
$(x^{\downarrow}(s[v-1]+1),\cdots,x^{\downarrow}(s[t]))$ and 
$(x^{\uparrow}(s[t]+1),\cdots,x^{\uparrow}(s[w]))$. Theorem 1 and 2 of \cite{201404_VidJaiMac} shows 
that the optimal solution to $\textsf{Nested}(v,w)$, 
$x^{\star}=(x^{\star}(s[v-1]+1),\cdots,x^{\star}(s[w]))$, satisfies the following inequalities with 
$x(e)=x^{\star}(e)$ for all $e \in \{s[v-1]+1,\cdots,s[w]\}$.
\begin{align}
 \label{eqn:downarrow} 0 \leq x(e)&\leq x^{\downarrow}(e) ~~~~ e \in \{s[v-1]+1,\cdots,s[t]\}, \\
 \label{eqn:uparrow} \beta(e) \geq x(e) &\geq x^{\uparrow}(e)  ~~~~ e \in \{s[t]+1,\cdots,s[w]\}.
\end{align}
It is easy to see that any $x= (x(s[v-1]+1),\cdots,x(s[w]))$ that satisfy (\ref{eqn:downarrow}), 
(\ref{eqn:uparrow}), and (\ref{eqn:equalityConstraint3}) is a feasible solution to 
$\textsf{Nested}(v,w)$. Also, $x^{\star}$ satisfy these three constraints as observed earlier. This 
shows that $x^{\star}$ can be obtained by solving $\textsf{RAP}(v,w)$ with the following bounds on 
variables.
 \begin{align}
  \label{eqn:downarrow2} \hat{c}_e=0,~&\hat{d}_e= x^{\downarrow}(e) ~~~~ e \in 
\{s[v-1]+1,\cdots,s[t]\}, \\
 \label{eqn:uparrow2} \hat{c}_e = x^{\uparrow}(e),~& \hat{d}_e = \beta(e)  ~~~~ e \in 
\{s[t]+1,\cdots,s[w]\}.
 \end{align}
$\textsf{Nestedsolve}(v,w)$ is described below.
\begin{algorithm}${\sf Nestedsolve}(v,w)$
\begin{enumerate}
 \item {\bf if} $v=w$ {\bf then}
      \begin{itemize}
      \item[] $(x(s[v-1]+1,\cdots,s[v]) ={\sf RAP}(v,v)$
     \end{itemize}
\item[] {\bf else}
      \begin{itemize}
      \item[] $t\leftarrow \left\lfloor\frac{v+w}{2} \right\rfloor$
      \item[] $(x(s[v-1]+1),\cdots,x(s[t])) = {\sf Nestedsolve}(v,t)$
      \item[] $(x(s[t]+1),\cdots,x(s[w])) = {\sf Nestedsolve}(t+1,w)$
      \item[] {\bf for} $i=s[v-1]+1$ to $s[t]$ {\bf do}
             \begin{itemize}
               \item[] $(\hat{c}_i,\hat{d}_i)=(0,x(i))$
             \end{itemize}
       \item[] {\bf for} $i=s[t]+1$ to $s[w]$ {\bf do}
             \begin{itemize}
               \item[] $(\hat{c}_i,\hat{d}_i)=(x(i),\beta(i))$
             \end{itemize}
       \item[] $(x(s[v-1]+1),\cdots,s[w]) ={\sf RAP}(v,w)$
      \end{itemize}  
\end{enumerate}

\end{algorithm}

The initial tightening of the constraints and initialization steps takes $\mathcal{O}(n)$ steps. The 
main recursive procedure, $\textsf{Nestedsolve}(1,m)$ is solved as a hierarchy of $\textsf{RAP}$ 
sub-problems, with $h=1+\left \lceil \log\,m\right\rceil $ levels of recursion. Each level $i \in 
\{1,2,\cdots,h\}$ has $2^{h-i}$ $\textsf{RAP}$ sub-problems. A $\textsf{RAP}(v,w)$ sub-problem is 
solved by the method of Frederickson and Johnson \cite{1982xx_FreJoh} that works in 
$\mathcal{O}(n\cdot\log(B/n))$ steps. A straightforward calculation shows that 
$\textsf{Nestedsolve}(1,m)$ works in $\mathcal{O}(n\cdot\log\,m\cdot\log(B/n))$ steps.
\section{A Taut-String Solution in a Special Case} \label{sec:string}
In this section, we consider the minimization of an interesting subclass of separable convex 
functions known as $d\operatorname{-}$separable convex functions \cite{1971xxMS_Vei} subject to 
ascending constraints (\ref{eqn:ladderConstraint})-(\ref{eqn:equalityConstraint}). We will begin 
with a slightly more general setting. Let $(E,g)$ be a polymatroid. Let $\phi:\mathbb{R} \rightarrow 
\mathbb{R}$ be a continuously differentiable and strictly convex function. Fix $d = 
(d_1,d_2,\cdots,d_n) \in \mathbb{R}^{E}_{+}$.

The objective function $W$ of $\Pi_1$ now has $w_e(x(e))=d_e\cdot\phi\left(\frac{x(e)}{d_e}\right)$. 
Such a $W$ is called $d\operatorname{-}$separable.

We now state and prove a result that, for a fixed $d$, the minimizer is the same for any $\phi$ that 
is continuously differentiable and strictly convex. Further, the minimizer has a special structure.

\begin{lemma}\label{lemma:lexopt}
 Let $W$ in (\ref{eqn:separableObjective}) be separable with 
$w_e(x(e))=d_e\cdot\phi\left(\frac{x(e)}{d(e)}\right)$, where $\phi$ is continuously differentiable 
and strictly convex. The $x^{\star}$ that minimizes $W$ over the bases of the polymatroid $(E,g)$ is 
the lexicographically optimal base of the polymatroid with respect to the weight vector $d$.
\end{lemma}
\begin{proof}
$W$ in (\ref{eqn:separableObjective}) is a separable convex function with 
\[w_e(x(e))= d_e\cdot\phi\left(\frac{x(e)}{d_e}\right).\]
Let $h_e$ be the derivative of $w_e$. By Theorem \ref{thm:opt2}, the minimizer of 
(\ref{eqn:separableObjective}) is the $h\operatorname{-}$lexicographically optimal base of the 
polymatroid $(E,g)$, where
we have
\[h_e(x(e))= \phi^{\prime}\left(\frac{x(e)}{d_e}\right),~~~~e \in E.\]
Since $h_e=\phi^{\prime}$ for all $e$ and $\phi^{\prime}$ is increasing, 
$h\operatorname{-}$lexicographically optimal base is the lexicographically optimal base with respect 
to the weight vector $d$. 
\end{proof}

Given the flexibility in $\phi$, let us choose $\phi=(u^2+1)^{1/2}$ (Veinott Jr. 
\cite{1971xxMS_Vei}). Then 
\begin{equation}\label{eqn:quadveinott}
  W(x)= \sum_{i=1}^{n}d_i\left(\left(\frac{x(i)}{d_i}\right)^{2}+1\right)^{1/2}= 
\sum_{i=1}^{n}\left(x(i)^{2}+d_i^{2}\right)^{1/2}
\end{equation}
Let us further restrict attention to ascending constraints of 
(\ref{eqn:ladderConstraint})-(\ref{eqn:equalityConstraint}). Then the objective function 
(\ref{eqn:quadveinott}) and constraints (\ref{eqn:ladderConstraint})-(\ref{eqn:equalityConstraint}) 
have the following geometric interpretation. 
Let $D_i= \sum_{e=1}^{i}d_e$, $E_i=\sum_{e=1}^{i}\alpha(e)$ and $X_i=\sum_{e=1}^{i}x(e)$. Also, let 
$D_0=E_0=0$. The constraints (\ref{eqn:ladderConstraint}) and (\ref{eqn:equalityConstraint}) are 
then 
\begin{eqnarray}
 X_i &\geq& E_i, ~~~~i=1,2,\cdots,n-1\nonumber\\
 X_n &=&E_n.\nonumber
\end{eqnarray}
A vector $x$ is feasible iff $(D_i,X_i)$ lies above $(D_i,E_i)$  in $x\operatorname{-}y$ plane for 
$i=1,2,\cdots,n-1$. Also, $(D_0,X_0)$ and $(D_n,X_n)$ must coincide with $(D_0,E_0)$ and 
$(D_n,E_n)$, respectively. Define feasible path as the path formed by the line segments joining 
$(D_{i-1},X_{i-1})$ and $(D_{i},X_{i})$ for $i=1,2,\cdots,n$ for a feasible $X_1,X_2,\cdots,X_n$. 
The length of the path corresponding to a feasible point $X_1,X_2,\cdots,X_n$ gives the value of the 
objective function in (\ref{eqn:quadveinott}) at the point. It is now obvious that the following 
{\it taut-string method} finds the minimum length path among all feasible paths, and hence the 
optimal solution, as described below.
\begin{itemize}
 \item Place pins at the points $(D_i,E_i)$, $i=0,1,\cdots,n$.
 \item Tie a string to the pin at origin and run the string above the points $(D_i,E_i)$, 
$i=1,2,\cdots,n$.
 \item Pull the string tight. The string traces the minimum length path from the origin to the point 
$(D_n,E_n)$.
\end{itemize}
\vspace*{2mm}
\begin{figure}[ht]
\centering
\includegraphics[scale=0.4]{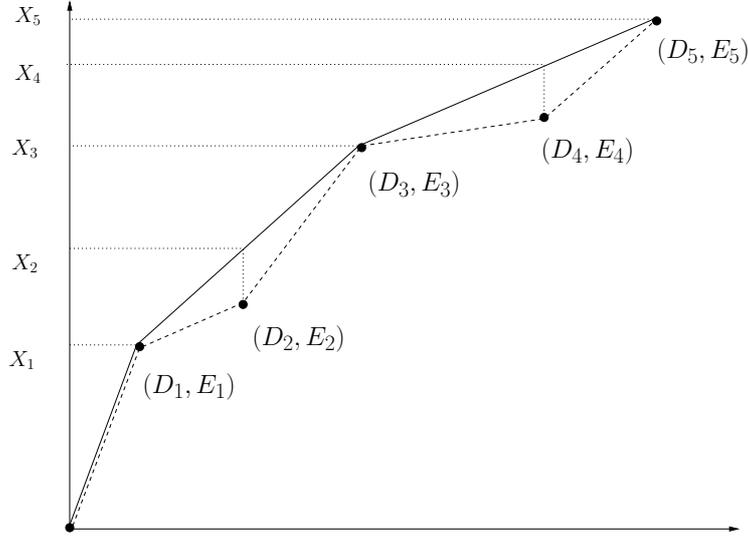}
\caption{The taut string solution.}
\label{fig:tss}
\end{figure}
\vspace*{2mm}

 Figure \ref{fig:tss} shows the {\it taut-string solution} for the points $(D_i,E_i)$, 
$i=1,2,\cdots,5$. All the paths starting from the origin and ending at $(D_5,E_5)$ that lie above 
the broken line correspond to feasible paths. The bold line in the figure traces the path of the 
tightened string. Since $E_i$ is increasing in $i$, $X_i$ corresponding to the shortest path is 
also 
increasing as can be observed from the figure and hence the optimal $x$ belongs to 
$\mathbb{R}^{E}_{+}$. The {\it taut string method} solves the special case of minimization of a 
$d\operatorname{-}$separable objective subject to ascending constraints.

It is clear that the taut string traces the concave cover, denoted by $\mathcal{C}$, which has the 
following properties.
\begin{itemize}
\item $\mathcal{C}$ is piece-wise linear and concave.  
\item $\mathcal{C}(0)=0, \mathcal{C}(D_n) = E_n \mbox{ and } \mathcal{C}(D_k) \geq E_k, \forall k$.
 \item The slope change-points of the piece-wise linear function lie on a subset of the set of 
points $\left\{D_1,D_2,\cdots,D_{n-1}\right\}$ and $\mathcal{C}(D_i)=E_i$ if $D_i$ is a slopechange 
point.
\end{itemize}

Note that the concave cover of a set of points is completely specified by the slope change points.
The minimum length path is also the graph of the least concave majorant of the points $(D_i,E_i)$, 
$i=1,2,\cdots,n$. Veinott Jr.'s \cite{1971xxMS_Vei} computation of the optimal $x$ corresponding to 
the minimum length path requires the computation of the least concave majorant of the points 
$(D_i,E_i)$, $i=0,1,\cdots,n$. The algorithm of \cite{1971xxMS_Vei} runs in $\mathcal{O}(n^2)$ 
steps\footnote{Veinott Jr. \cite{1971xxMS_Vei} considered another problem as well, one with upper 
and lower bounds on the variables on top of the ascending constraints. Veinott Jr. 
\cite{1971xxMS_Vei} provided a taut string solution and a common algorithm for both problems. It is 
this common algorithm that takes $\mathcal{O}\left(n^2\right)$ steps.}. 

In the next section, we provide an $\mathcal{O}(n)$ algorithm for finding the concave cover of the 
set of points $(D_0,E_0),(D_1,E_1),(D_2,E_2),\cdots,(D_n,E_n)$.
function being weighted $\alpha\operatorname{-}$fair utility function,
$(\sum_{e=1}^{i}\alpha(e),\sum_{e=1}^{i}p_e^{1/\alpha}),~ i=1,2,\cdots,n$. This also motivates us to 
look for an efficient algorithm to find the concave cover of a set of points in $x\operatorname{-}y$ 
plane.

\section{String Algorithm}
\label{sec:fba}
In the previous section, we described the method proposed by Veinott Jr. that reduces problem $\Pi$ 
for the case of a $d\operatorname{-}$separable objective function to a geometrical problem of 
finding the concave cover of the set of points $(D_0,E_0),(D_1,E_1),(D_2,E_2),\cdots,(D_n,E_n)$ in 
$x\operatorname{-}y$ plane. Let the points be denoted as $t_0,t_1,t_2,\cdots,t_{n}$. In this 
section, we describe an algorithm that finds the concave cover of these points in $\mathcal{O}(n)$ 
steps. The algorithm is an adaptation of the String Algorithm of Muckstadt and Sapra 
\cite{2010xxSPR_MucSap} that finds a ``convex envelope'' of a set of points for an inventory 
management problem.

 Consider a piece-wise linear function with the slope change abscissa locations being a subset of 
$\left\{D_0,D_1,D_2,\cdots,D_n\right\}$. The piece-wise linear function is concave if the slopes of 
the adjacent straight line segments are decreasing; more precisely, if $t_i=(D_i,E_i)$ corresponds 
to a $D_i$ where slope changes, then 
\begin{eqnarray}
 \mbox{Slope of } t_i-t_{LN(i)} & > \mbox{Slope of } t_{RN(i)}-t_i,\label{eqn:concavitycond}
\end{eqnarray}
where $LN(i)$ \mbox{is the index of the slope change point closest to the left of }$D_i$, and 
$RN(i)$ \mbox{is the index of the slope change point closest to the right of }$D_i$. See figure 
\ref{fig:figa} taking $i=1$.

The algorithm takes as input $t_0,t_1,\cdots,t_n$. The algorithm checks the concavity condition in 
(\ref{eqn:concavitycond}) at every point. If the condition is not satisfied at a point $t_i$, then 
$t_i$ is dropped. After performing a sequence of such eliminations, the algorithm puts out the slope 
change points of the desired concave cover. The algorithm is formally described below.

\begin{algorithm}\label{alg:fb} {\sf String Algorithm}
 \begin{enumerate}
  \item \textbf{Input}: A sequence of points $t_0,t_1,\cdots,t_n$ with increasing (or nondecreasing 
abscissas). 
  \item \textbf{Initialize}:$j=0$, and $LN(i)=i-1$, for $i=1,2,\cdots,n$ and $RN(i)=i+1$, for 
$i=0,1,\cdots,n-1$.	
  \item \textbf{Forward Step}:
       \begin{itemize}
       \item [] $j\leftarrow j+1$.
       \item [] \textbf{if} $j\neq n$ 
                 \begin{itemize}
                 \item[] $i= j$.
                 \item[] Go to \textbf{Backward Step}.
                 \end{itemize}
      \item [] \textbf{else}
                       \begin{itemize}
                        \item [] \textbf{return} $RN$.
                        \end{itemize}
        \end{itemize}

\item \textbf{Backward Step}: 
    \begin{itemize}

    \item [] \textbf{if}  condition (\ref{eqn:concavitycond}) is satisfied at point $t_i$
         \begin{itemize}
         \item[] go to \textbf{Forward Step}.
         \end{itemize}
    \item [] \textbf{else}
         
         \begin{itemize}
           \item [] $LN(RN(i))= LN(i),RN(LN(i)) = RN(i)$,$RN(i)=$\textbf{null}.
         \item [] \textbf{if} $LN(i) \neq 0$
            \begin{itemize} 
             \item [] $i\leftarrow LN(i)$ \item[] go to \textbf{Backward Step}.
            \end{itemize} 
      \item [] \textbf{else} 
           \begin{itemize} 
            \item [] go to \textbf{Forward Step}.
           \end{itemize}
      \end{itemize}
 \end{itemize}
\end{enumerate}
\end{algorithm}

\begin{figure}[h!]
\centering
\begin{tabular}{cc}
\subfloat[Backward step at 
$j=1$]{\label{fig:figa}\includegraphics[width=7cm,height=6cm]{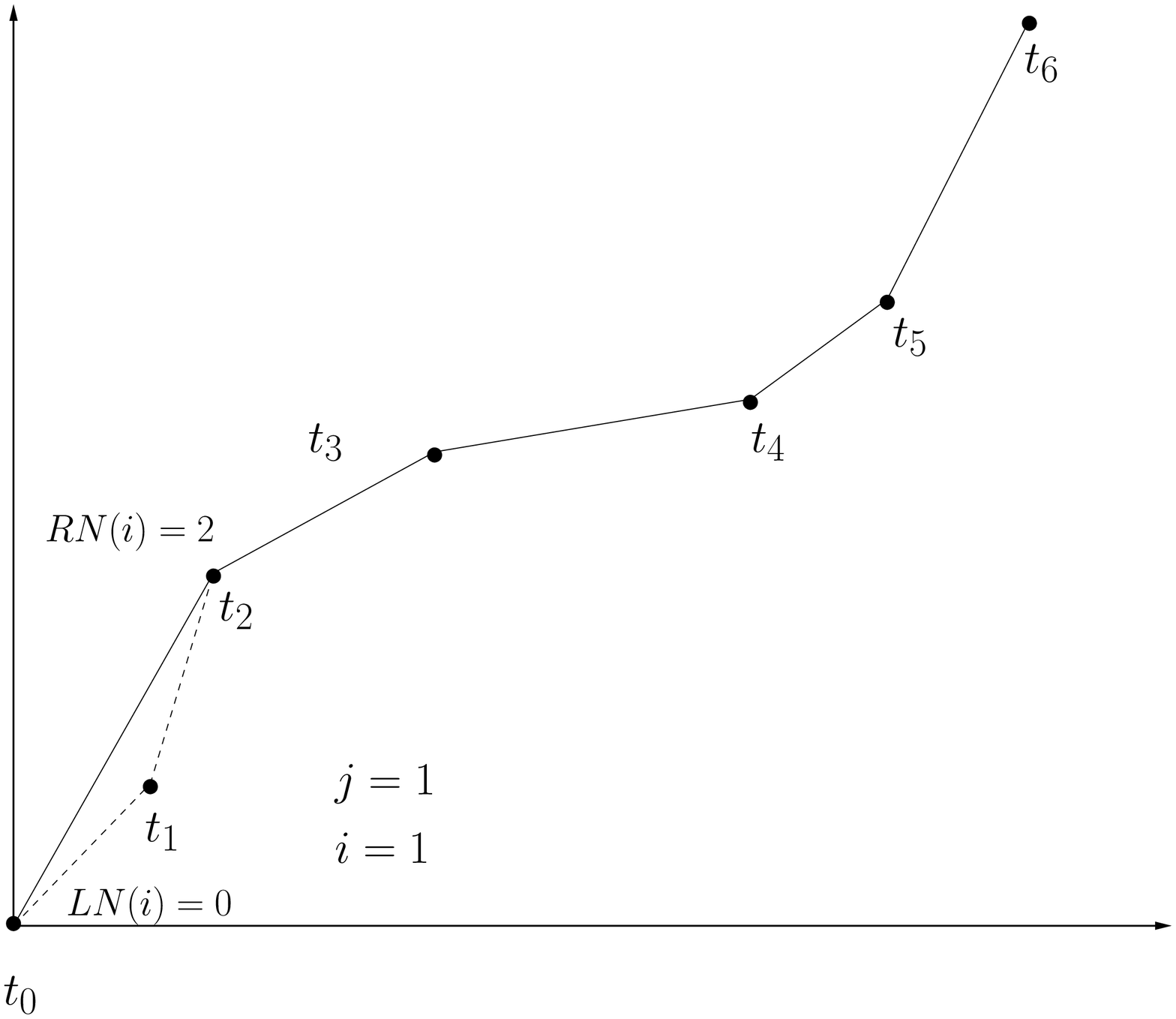}}&
\subfloat[Backward step at 
$j=4$]{\label{fig:figb}\includegraphics[width=7cm,height=6cm]{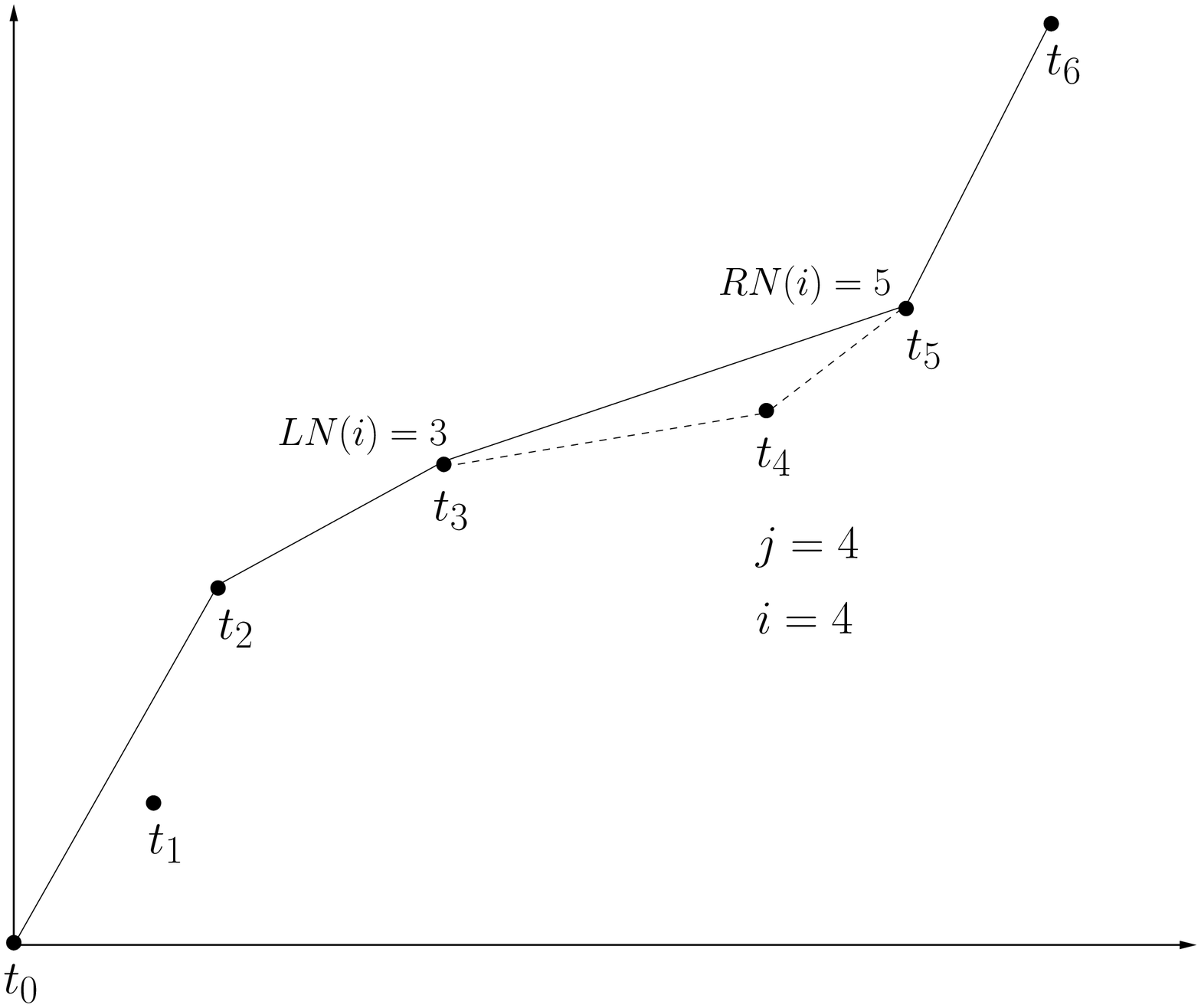}} \\ 
\subfloat[Backward step at $j=5$ at node 
$t_5$]{\label{fig:figc}\includegraphics[width=7cm,height=6cm]{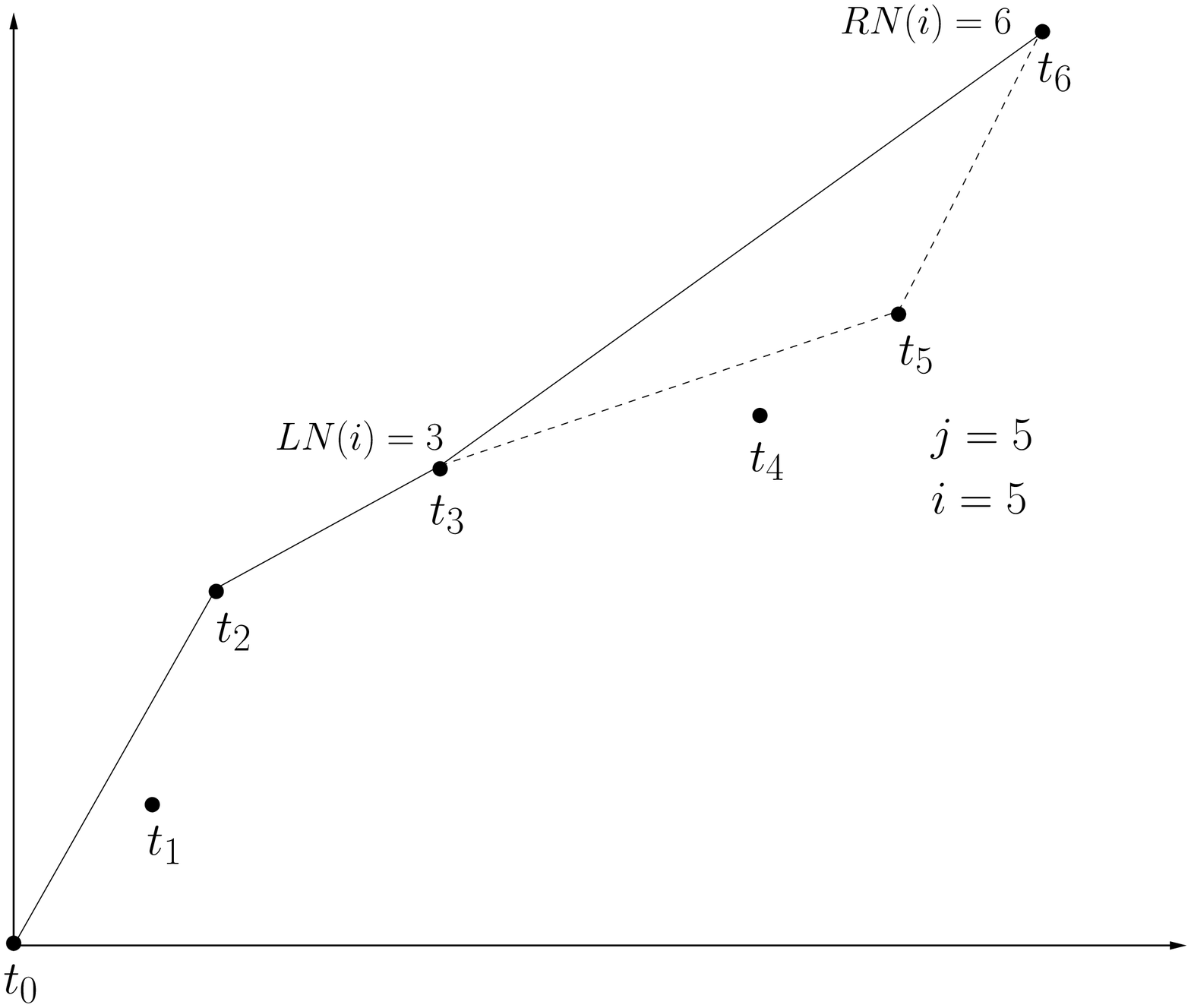}} &
\subfloat[Backward step at $j=5$ at node 
$t_3$]{\label{fig:figd}\includegraphics[width=7cm,height=6cm]{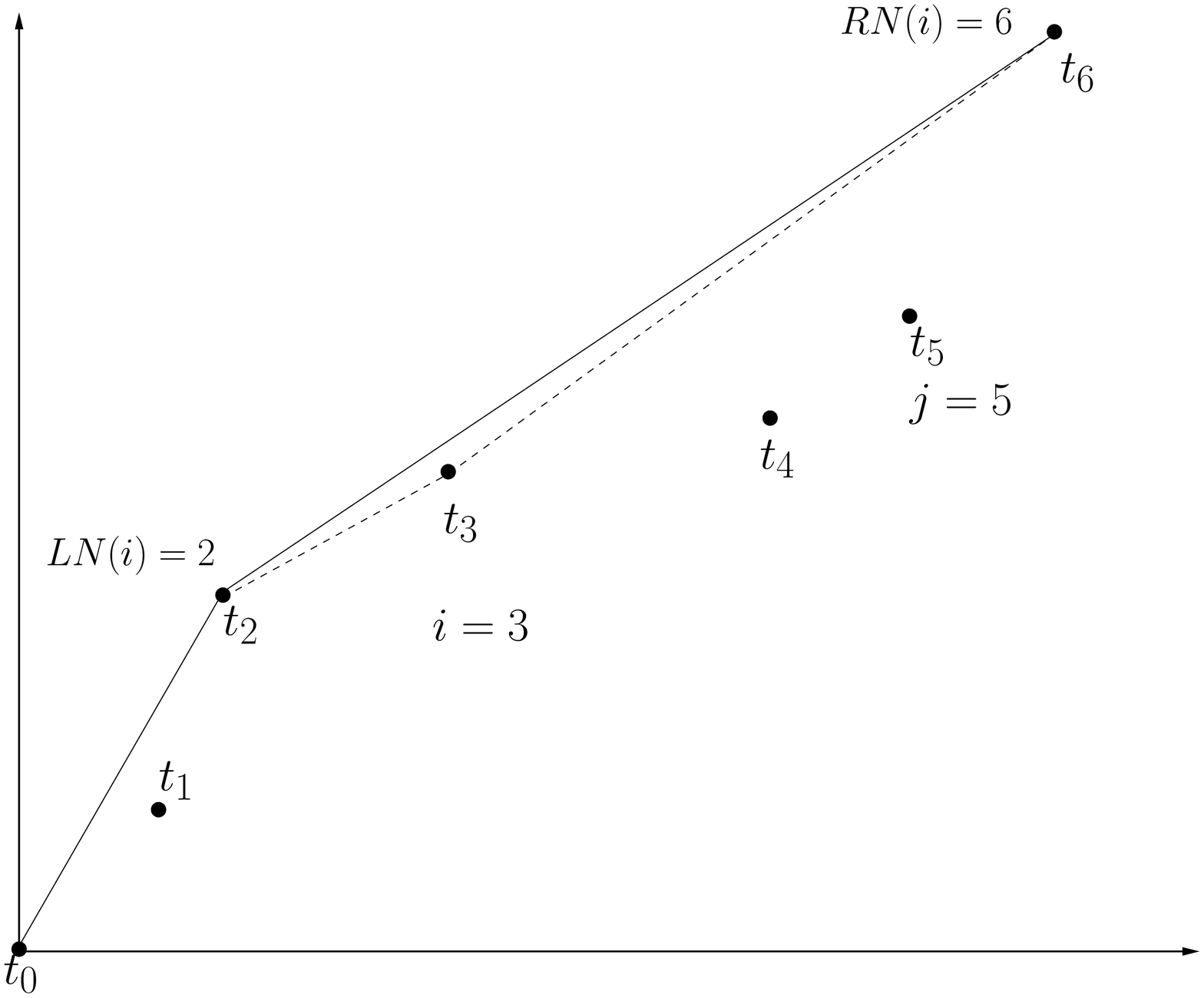}} \\
\end{tabular}
\caption{String algorithm}
\label{fig:fig}
\end{figure}
Figure \ref{fig:fig} shows the key steps of the String Algorithm for the case $n=6$. Figure 
\ref{fig:figb} shows the backward step of the string algorithm at $t_4$. As the concavity condition 
in (\ref{eqn:concavitycond}) is violated  at $t_4$, the point is dropped. Now, the left adjacent 
slope change point of $D_5$ is $D_3$ and the right adjacent point of $D_3$ is $D_5$. These changes 
are made in the backward step of the algorithm by updating the $LN$ and $RN$ pointers. 

At the completion of the algorithm, the indices corresponding to the valid entries (entries 
excluding {\bf null}) in $RN$ are the indices of the slope change points of the concave cover. This 
is straightforward to check and a formal proof is omitted. The piece-wise linear function formed by 
bold lines in the figure \ref{fig:figd} is the concave cover of the set of points 
$t_0,t_1,t_2,\cdots,t_6$. All points except $t_2$ are dropped and $D_2$ is the only slope change 
point for the concave cover.

The following lemma gives the complexity of String Algorithm.

\begin{prop}
 The complexity of the String Algorithm is $\mathcal{O}(n)$.
\end{prop}
\begin{proof}
The number of executions of the forward step is $n$, each consisting of constant number of 
operations. Hence the complexity of forward procedure is $\mathcal{O}(n)$. The number of times the 
backward step is executed is equal to the number of points dropped which is less than $n$. In a 
single execution of the backward step, the concavity condition is checked. If a point is dropped, 
then the values of $LN(RN(i))$ and $RN(LN(i))$ are modified. All these amount to constant number of 
operations which does not scale with $n$. Hence the total complexity of the backward procedure is 
also $\mathcal{O}(n)$. This completes the proof of the proposition.
\end{proof}

\section{Conclusion}
We discussed several algorithms that solve problem $\Pi$, a separable convex optimization problem 
with linear ascending constraints, that arises as a core optimization in several resource 
allocation 
problems. The algorithms can be classified as greedy-type or decomposition-type algorithms. The 
best 
in class algorithms have complexity 
$\mathcal{O}\left(n\cdot(\log\,n)\cdot\log\left(B/(n\epsilon)\right)\right)$ to get an 
$\epsilon\operatorname{-}$optimal solution, with $B$ being the total number of units to be 
allocated. We also considered a $d\operatorname{-}$separable objective. In this special case, the 
solution is the lexicographically optimal base of a polymatroid formed by the constraint set of 
$\Pi$. We then argued that finding the lexicographically optimal base is equivalent to finding the 
least concave majorant of a set of points on the $\mathbb{R}^{2}_{+}$ quadrant. We then described 
an 
$\mathcal{O}(n)$ algorithm for this problem. This is significant because of its applicability to 
the 
minimization of popular $d\operatorname{-}$separable functions such as $\alpha\operatorname{-}$fair 
utility functions which are widely used as utility functions in network utility maximization 
problems.

\bibliographystyle{siam}
\bibliography{IEEEabrv,wisl}

\appendix

\section{Proof of correctness and optimality}
In this section, we give the proof of Theorem \ref{th3}.
 \subsection{Feasibility}
We begin by addressing a necessary and sufficient condition for feasibility.

\vspace*{.15in}
\begin{lemma} \label{lemma:fb}
The feasible set is nonempty if and only if (\ref{eqn:commonSense1}) is satisfied.
\end{lemma}
\vspace*{.15in}

\begin{proof}
Assume that (\ref{eqn:commonSense1}) is not satisfied for some $l$ and let $l'$ be the smallest 
among such indices. This implies that even if we assign the largest possible value for $x(e)$, i.e., 
$x(e)=\beta(e)$ for $e=1,2,\cdots,l'$, the constraint (\ref{eqn:ladderConstraint}) for $l=l'$ cannot 
be  satisfied, and hence the constraint set is empty.

To prove sufficiency, assume (\ref{eqn:commonSense1}) holds. Let $l=l'$ be the smallest index $l$ 
for which
\begin{equation}
  \sum_{e=1}^l \beta(e) \geq \sum_{e=1}^n \alpha(e)
\end{equation}
holds. Now, assign $x(e)=\beta(e)$, for $e=1,2,\ldots,l'-1$, assign
$$x(l')=\sum_{e=1}^n \alpha(e)-\sum_{e=1}^{l'-1} \beta(e),$$
and $x(e)=0$ for $e=l'+1,l'+2,\cdots,n$. Clearly, $x$ satisfies 
(\ref{eqn:ladderConstraint})-(\ref{eqn:equalityConstraint}) and is therefore a feasible point. This 
proves the lemma.
\end{proof}

\subsection{Proof of Theorem \ref{th3}}
In order to prove Theorem \ref{th3}, the following should be shown to ensure that the algorithm 
terminates and generates the desired allocation.
 \begin{enumerate}
  \item The set whose minimum is taken in (\ref{eqn:gamma}) should be nonempty at each iteration 
step.
  \item The assignments in (\ref{assign1}) and (\ref{assign2}) should yield a feasible allocation at 
each iteration step, and the reduced problem for the next iteration is a similar but smaller 
problem.
  \item The output of the algorithm should be feasible.
  \item The output should satisfy the sufficiency conditions for optimality in Theorem 
\ref{thm:opt}.
\end{enumerate}

We begin by proving that the minimum in (\ref{eqn:gamma}) is over a nonempty set. For this, we need 
the following lemma.

\vspace*{.15in}
\begin{lemma} \label{lemma1}
 At any iteration step $j$, for any $l$ with $1 \leq l < s(j-1)$, there exists an $\eta_{l}^{j}$ 
satisfying (\ref{eqn:zero}) if and only if
\begin{equation}\label{eqn:l1-3}
  0 \leq \sum_{e=l}^{s(j-1)-1}\alpha(e) \leq  \sum_{e=l}^{s(j-1)-1}\beta(e).
\end{equation}
\end{lemma}
\vspace*{.15in}

\begin{proof}
 Let $ \overline{\eta} = \max{\{w_{e}^{-}(\beta(e)):l \leq e < s(j-1)}\}$.
From the definition of $H_{e}^{+}$, for all $\eta \geq \overline{\eta}$, we have 
\[H_{e}^{+}(\eta)=\beta(e)~~~~(l \leq e < s(j-1)),\]
and so
\begin{equation}\label{eqn:l1-1}
 \sum_{e=l}^{s(j-1)-1}H_{e}^{+}(\eta)=\sum_{e=l}^{s(j-1)-1}\beta(e)~~~(\eta \geq \overline{\eta}).
\end{equation}
Next, define $ \underline{\eta}= \min{\{w_{e}^{+}(0):l \leq e < s(j-1)\}}$.
The definition of $H_{e}^{-}$ implies that for all $\eta \leq \underline{\eta}$, we have 
\[H_{e}^{-}(\eta)=0~~~~( l \leq e < s(j-1)),\]
and so
\begin{equation}\label{eqn:l1-2}
 \sum_{e=l}^{s(j-1)-1}H_{e}^{-}(\eta)=0~~~(\eta \leq \underline{\eta}  ).
\end{equation}
Necessity of (\ref{eqn:l1-3}) is then obvious from (\ref{eqn:l1-1}) and (\ref{eqn:l1-2}), since 
$H_{e}^{+}$ and  $H_{e}^{-}$ are nondecreasing. For sufficiency, in addition to the nondecreasing 
nature of $H_{e}^{+}$ and  $H_{e}^{-}$, we also have  $H_{e}^{+}=H_{e}^{-}$ at all points of 
continuity of $H_{e}^{+}$, and  $H_{e}^{-}$ is the left continuous version of the right continuous 
$H_{e}^{+}$. Given (\ref{eqn:l1-3}) and these observations, it follows that we can find an $\eta$ 
that satisfies (\ref{eqn:zero}).
\end{proof}

\vspace*{.15in}
\begin{prop}
If the feasible region is nonempty, then at any iteration step $j$, the index $l=1$ satisfies 
(\ref{eqn:one}). Hence the set over which the minimum is taken in (\ref{eqn:gamma}) is nonempty.
\end{prop}
\vspace*{.15in}

\begin{proof}
If the feasible region is nonempty, we have from Lemma \ref{lemma:fb} that (\ref{eqn:commonSense1}) 
holds for $l = s(j-1)-1$, i.e.,
\[\sum_{e=1}^{s(j-1)-1}\alpha(e) \leq  \sum_{e=1}^{s(j-1)-1}\beta(e).\]
By Lemma \ref{lemma1}, $\eta_{l}^{j}$ exists. Consequently, the set over which the minimum is taken 
now contains $\eta_{1}^{j}$ and is therefore nonempty.
\end{proof}

\vspace*{.15in}
We shift attention to Subroutine \ref{subroutine1}.
We  show that the $s(j)$ put out by the Subroutine \ref{subroutine1} satisfies a property that is 
crucial to prove the feasibility of the output of Algorithm \ref{algorithm}.
This is the property that  the partial sums of $H_{e}^{+}(\Gamma_{j})$ from $s(j)$ to each of the 
tied indices exceeds the corresponding partial sums of $\alpha(e)$. Since we will show equality of 
the constraints at $s(j)$, the above property is necessary for feasibility.

\vspace*{.15in}
\begin{prop}\label{prop:3}
The index $s(j)$ chosen by Subroutine \ref{subroutine1} satisfies the following property:

Let $l_{1},l_{2},\cdots,l_{r}$  be the indices that attain the minimum in (\ref{eqn:gamma}) with 
$s(j-1) > l_{1} > l_{2} \ldots > l_{r} \geq 1.$
Let $l_{p+1}$ with $1 \leq p+1 \leq r$ be the index chosen by Subroutine \ref{subroutine1} as 
$s(j)$. Denote $I=\{l_{1},l_{2},\ldots,l_{p}\}$. (When $p = 0$, $I$ is empty). Then
\begin{equation} \label{eqn:p3a}
 \sum^{l-1}_{e=s(j)}H_{e}^{+}(\Gamma_{j}) \geq \sum^{l-1}_{e=s(j)}\alpha(e)~~~~~~~~~(l \in I).
\end{equation}
\end{prop}
\vspace*{.15in}

\begin{proof}
If $s(j)=l_{r}$ at the end of Subroutine \ref{subroutine1}, i.e., the iterate $t$ did not decrement 
at all, it follows that (\ref{eqn:subroutine1}) holds for all $k \in \{l_{1},l_{2},\ldots,l_{r-1}\} 
\cup \{s(j-1)\}$ and $m=l_{r}$. Hence (\ref{eqn:p3a}) is trivially true.

Now suppose $s(j) > l_{r}$. In Subroutine \ref{subroutine1}, the variables $i$ and $t$ have initial 
values $1$ and $r$, respectively.
If (\ref{eqn:subroutine1}) is satisfied for $m=l_{r}$ and $k=l_{i}$, then the value of $i$ is 
increased by unity, until (\ref{eqn:gamma}) is not satisfied at some $i=p'$ and $t=r$, i.e.,
\begin{equation} \label{eqn:p3b}
 \sum_{e=l_{r}}^{l_{p'}-1}H_{e}^{+}(\Gamma_{j}) < \sum_{e=l_{r}}^{l_{p'}-1}\alpha(e),
\end{equation}
and therefore (\ref{eqn:subroutine1}) is satisfied for all $s$ satisfying $1 \leq s < p'$, i.e.,
\begin{equation} \label{eqn:p3c}
 \sum_{e=l_{r}}^{l_s-1}H_{e}^{+}(\Gamma_{j}) \geq \sum_{e=l_{r}}^{l_s-1}\alpha(e)~~~(1 \leq s < p').
\end{equation}
Moreover,  $l_{p'} \geq s(j)$. If (\ref{eqn:subroutine1}) is not satisfied already for $m=l_{r}$ and 
$k=l_{i}$, then $p'=1$ and (\ref{eqn:p3c}) is irrelevant. Now the algorithm reduces the value of $t$ 
in steps of unity until (\ref{eqn:subroutine1}) is satisfied for $m=l_{r'}$ and $k=l_{p'}$, i.e.,
\begin{equation}\label{eqn:p3d}
 \sum_{e=l_{r'}}^{l_{p'}-1}H_{e}^{+}(\Gamma_{j}) \geq \sum_{e=l_{r'}}^{l_{p'}-1}\alpha(e),
\end{equation}
with $l_{r'} \leq s(j)$. From (\ref{eqn:p3b}) and (\ref{eqn:p3c}), we get (for $p' > 1$)
\begin{equation}\label{eqn:p3e}
 \sum_{e=l_{p'}}^{l_s-1}H_{e}^{+}(\Gamma_{j}) \geq \sum_{e=l_{p'}}^{l_s-1}\alpha(e)~~~(1 \leq s < 
p').
\end{equation}
Summing (\ref{eqn:p3d}) and (\ref{eqn:p3e}) when $p' > 1$, or simply considering (\ref{eqn:p3d}) 
when $p' = 1$, we get
\begin{equation} \label{eqn:p3f}
  \sum_{e=l_{r'}}^{l_s-1}H_{e}^{+}(\Gamma_{j}) \geq \sum_{e=l_{r'}}^{l_s-1}\alpha(e)~~~(1 \leq s 
\leq p').
\end{equation}
But (\ref{eqn:p3f}) is just (\ref{eqn:subroutine1})  for $m=l_{r'}$ and $k=l_{s}$ with $1 \leq s 
\leq p'$. Proceeding by induction, $r'$ decrements, $p'$ increments, and eventually they coincide at 
$s(j)$, and (\ref{eqn:p3a}) holds for all $1 \leq l < s(j-1)$. Validity of (\ref{eqn:p3a}) for 
$l=s(j-1)$ is clear from the definition of  $\eta_{s(j)}^{j}$.
\end{proof}
\vspace*{.15in}

The following lemma is a corollary to Proposition \ref{prop:3} and is useful in proving optimality.

\vspace*{.15in}
\begin{lemma}\label{l4}
The sequence $\{\Gamma_{i}:i=1,2,3,\ldots,k\}$ put out by Algorithm \ref{algorithm} satisfies 
$\Gamma_{1} < \Gamma_{2} < \cdots < \Gamma_{k}$.
\end{lemma}
\vspace*{.15in}

\begin{proof}
We will prove that at iteration $j+1$, the number $\eta_{l}^{j+1}$ which is the (smallest) solution 
(if it exists) of
\begin{equation}
  \label{eqn:He-eta}
  \sum_{e=l}^{s(j)-1}H_{e}^{-}(\eta) \leq \sum_{e=l}^{s(j)-1}\alpha_{e} \leq 
\sum_{e=l}^{s(j)-1}H_{e}^{+}(\eta)
\end{equation}
for $l$ satisfying $1 \leq l \leq s(j)$, is strictly greater than $\Gamma_{j}$. Hence their minimum 
$\Gamma_{j+1}$ is also strictly greater than $\Gamma_{j}$, and the proof will be complete.

For  indices $l$ that satisfy (\ref{eqn:zero}) in the $j^{th}$ iteration with $l < s(j) $ and 
$\eta_{l}^{j}= \Gamma_{j}$ (i.e., $l$ is a tied index), there is some $l'$ satisfying $s(j-1) > l' 
\geq s(j)$ and
\begin{equation} \label{l4a}
 \sum^{l'-1}_{e=l}H_{e}^{+}(\Gamma_{j}) < \sum^{l'-1}_{e=l}\alpha(e);
\end{equation}
otherwise index $t$ in Subroutine \ref{subroutine1} would have pointed to $s(j)$, and an $s(j) > l$ 
would not have been picked by the subroutine. But from Proposition \ref{prop:3}, we also have
\begin{equation}\label{l4b}
 \sum^{l'-1}_{e=s(j)}H_{e}^{+}(\Gamma_{j}) \geq \sum^{l'-1}_{e=s(j)}\alpha(e).
\end{equation}
Subtract (\ref{l4b}) from (\ref{l4a}) to get
\begin{equation}\label{l4e}
 \sum_{e=l}^{s(j)-1}H_{e}^{+}(\Gamma_{j}) < \sum_{e=l}^{s(j)-1}\alpha(e).
\end{equation}
Hence $\eta_{l}^{j+1} > \Gamma_{j}$ for all such tied $l$.
For all nontied indices $l < s(j)$, i.e., indices that satisfy (\ref{eqn:zero}) in $j^{th}$ 
iteration but with $\eta_{l}^{j} > \Gamma_{j}$ or $\eta_{l}^{j}$ does not exist, we must have
\begin{equation}\label{l4c}
 \sum^{s(j-1)-1}_{e=l}H_{e}^{+}(\Gamma_{j}) < \sum^{s(j-1)-1}_{e=l}\alpha(e),
\end{equation}
and therefore after noting (\ref{l4e}) for the tied indices, we conclude that (\ref{l4c}) holds for 
all $l < s(j)$.
Furthermore, since $s(j)$ is a tied index and $\eta_{s(j)}^l = \Gamma_j$, we also have
\begin{equation}\label{l4d}
 \sum^{s(j-1)-1}_{e=s(j)}H_{e}^{+}(\Gamma_{j}) \geq \sum^{s(j-1)-1}_{e=s(j)}\alpha(e).
\end{equation}
Subtract (\ref{l4d}) from (\ref{l4c}) to get
\begin{equation}
 \sum^{s(j)-1}_{e=l}H_{e}^{+}(\Gamma_{j}) < \sum^{s(j)-1}_{e=l}\alpha(e)~~~(l < s(j)).
\end{equation}
Since $H_e^+$ is nondecreasing, the solution $\eta_l^{j+1}$ to (\ref{eqn:He-eta}) must be strictly 
larger than $\Gamma_j$, i.e., $\eta_{l}^{j+1} > \Gamma_{j}$, if the solution exists.
\end{proof}
\vspace*{.15in}

We next move to Subroutine \ref{subroutine2}. This subroutine assigns values to variables (from 
higher indices to lower indices) in stages over possibly several iterations. After each iteration, 
the assignment problem reduces to a similar but smaller problem. The next lemma is a step  to say 
that every index in between $s(j)$ and $s(j-1)$, both included, can be assigned successfully within 
Subroutine \ref{subroutine2} without a need to execute Subroutine~\ref{subroutine1} after every 
substage. In particular, $s(j)$ and $\Gamma_{j}$ would remain stable over the entire execution of 
Subroutine~\ref{subroutine2}.

\vspace*{.15in}
\begin{lemma}\label{l5}
In  step $5$ of Subroutine~\ref{subroutine2}, the indices satisfying $s(j) < l < l_{t}^{m}$ and 
(\ref{eqn:srt27})  is a subset of $I_{m+1}=I \cap \{l:l < l^{m}_{t}\}$. Hence the set $I_{m+1}$ 
contains all the indices $l < l^{m}_{t}$ that satisfy (\ref{eqn:srt27}).
\end{lemma}
\vspace*{.15in}

\begin{proof}
For any index $l'$ satisfying $s(j) <l' < l_{t}^{m}$ with $l' \notin I \cap \{ l:l<l_{t}^{m}\} $, 
either $\eta_{l'}^{j}$ does not exist or $l'$ is not a tied index, i.e.,  $\eta_{l'}^{j}$ exists but 
satisfies $\eta_{l'}^{j} > \Gamma_{j}$.
In the former case, by Lemma \ref{lemma1} and the fact that $\alpha(e)$'s are positive, we have
\begin{equation}
 \sum_{e=l'}^{s(j-1)-1}\beta(e) < \sum_{e=l'}^{s(j-1)-1}\alpha(e)
\end{equation}
and so
\begin{equation}
 \sum_{e=l'}^{s(j-1)-1}H_{e}^{+}(\Gamma_{j}) \leq \sum_{e=l'}^{s(j-1)-1}\beta(e) < 
\sum_{e=l'}^{s(j-1)-1}\alpha(e).
\end{equation}
In the latter case, we must have
\begin{equation}\label{l5a}
\sum_{e=l'}^{s(j-1)-1}H_{e}^{+}(\Gamma_{j}) < \sum_{e=l'}^{s(j-1)-1}\alpha(e)
\end{equation}
for otherwise $\Gamma_{j}$ would be a strictly smaller choice for $\eta^{j}_{l'}$, contradicting the 
choice of $\eta_{l'}^j$. So, in either case, we have that (\ref{l5a}) holds. But we also have
\begin{equation}
\sum_{e=l^{m}_{t}}^{s(j-1)-1}H_{e}^{+}(\Gamma_{j}) \geq   \sum_{e=l^{m}_{t}}^{s(j-1)-1}\alpha(e).
\end{equation}
Subtracting the two inequalities, we get
\begin{equation}
 \sum_{e=l'}^{l^{m}_{t}-1}H_{e}^{+}(\Gamma_{j}) <   \sum_{e=l'}^{l^{m}_{t}-1}\alpha(e).
\end{equation}
i.e., $l'$ will not satisfy (\ref{eqn:srt27}), which is what we set out to prove.
\end{proof}

\vspace*{.15in}
The following lemma is a key inductive  step to show that the assignment problem in step 3 of 
Subroutine~\ref{subroutine2} reduces the problem to a similar but smaller problem after each 
iteration.

\vspace*{.15in}
\begin{lemma}\label{l6}
Suppose at the $m^{th}$ iteration in Subroutine~\ref{subroutine2}, we have indices 
$l_{0}^{m},l_{1}^{m},\cdots,l_{p_{m}}^{m},l_{p_{m}+1}^{m}=s(j)$. Let 
$M_{m}=\{s(j),s(j)+1,\cdots,l_{0}^{m}-1\}$ and with $I_{m}=\{l_{1}^m,l_{2}^m,\ldots,l_{p_{m}}^{m} 
\}$, let $I_{m} \cup \{s(j)\}$ be the indices $l$ that satisfy
\begin{equation}\label{l6a}
  \sum_{e=l}^{l^{m}_{0}-1}H_{e}^{-}(\Gamma_{j}) \leq \sum_{e=l}^{l^{m}_{0}-1}\alpha(e) \leq 
\sum_{e=l}^{l^{m}_{0}-1}H_{e}^{+}(\Gamma_{j}).
\end{equation}
Note that $I_{m} \subseteq M_{m}$.
For all indices $l \in M_{m} \backslash (I_{m}\cup \{s(j)\})$, we then have
\begin{equation}\label{l6b}
 \sum_{e=l}^{l^{m}_{0}-1}H_{e}^{+}(\Gamma_{j}) < \sum_{e=l}^{l^{m}_{0}-1}\alpha(e).
\end{equation}
Let $M_{m+1}$ be the set of indices in $M_{m}$ corresponding to $x(e)$'s that are not assigned in 
the $m^{th}$ iteration and let $I_{m+1}$ be the set obtained in step 5 of 
Subroutine~\ref{subroutine2}. The set $I_{m+1}$ and $M_{m+1} \backslash (I_{m+1}\cup \{s(j)\})$ 
satisfy properties (\ref{l6a}) and (\ref{l6b}), respectively, with $m$ replaced by $m+1$.
\end{lemma}
\vspace*{.15in}

\begin{proof}
 Note that from step 2 of Subroutine~\ref{subroutine2}, $l_{t}^{m}$ is chosen so that
\begin{equation}
\gamma =  
\sum_{e=l^{m}_{t}}^{l^{m}_{0}-1}\alpha(e)-\sum_{e=l^{m}_{t}}^{l^{m}_{0}-1}H_{e}^{-}(\Gamma_{j}).
\end{equation}
With $l_{0}^{m+1}=l_{t}^{m}$, we also have
\begin{eqnarray*}
  M_{m+1} & = & \{ s(j),s(j)+1,\ldots,l_{0}^{m+1}-1\}, \\
  I_{m+1} & = & \{ l: l \in I,  l<l^{m}_{t}, \mbox{ and satisfies (\ref{eqn:srt27})}\}.
\end{eqnarray*}
We claim that
\begin{equation}
  \label{eqn:He+}
 \sum_{e=l}^{l^{m}_{t}-1}H_{e}^{+}(\Gamma_{j}) \geq \sum_{e=l}^{l^{m}_{t}-1}\alpha(e)~~~(l \in 
I_{m+1} \cup \{s(j)\}).
\end{equation}
Indeed, the inequalities hold true for $l \in I_{m+1}$  by construction and for $l = s(j)$ by 
Proposition \ref{prop:3}. Furthermore
\begin{equation}
\gamma =  
\sum_{e=l^{m}_{t}}^{l^{m}_{0}-1}\alpha(e)-\sum_{e=l^{m}_{t}}^{l^{m}_{0}-1}H_{e}^{-}(\Gamma_{j}) \leq 
\sum_{e=l}^{l^{m}_{0}-1}\alpha(e)-\sum_{e=l}^{l^{m}_{0}-1}H_{e}^{-}(\Gamma_{j})~~~(l \in I_{m+1} 
\cup \{s(j)\})
\end{equation}
which on rearrangement yields
\begin{equation}\label{l6c}
 \sum_{e=l}^{l^{m}_{t}-1}H_{e}^{-}(\Gamma_{j}) \leq \sum_{e=l}^{l^{m}_{t}-1}\alpha(e) ~~~(l \in 
I_{m+1} \cup \{s(j)\}).
\end{equation}
Since $l_{t}^{m}=l_{0}^{m+1}$, (\ref{l6c}) and (\ref{eqn:He+}) yield (\ref{l6a}) with $m+1$ in place 
of $m$. The analogue
 of property (\ref{l6b}) for $l \in M_{m+1} \backslash (I_{m+1}\cup \{s(j)\})$ follows from Lemma 
\ref{l5}.
\end{proof}

\vspace*{.15in}
\begin{lemma}\label{l7}
 Assume that $I_{m}$ and $M_{m}$ are as in Lemma \ref{l6} and such that (\ref{l6a}) holds for 
indices in $I_{m} \cup \{s(j)\}$ and (\ref{l6b}) for indices in $M_{m} \backslash (I_{m} \cup 
\{s(j)\})$. The assignments in (\ref{assign1}) and  (\ref{assign2}) in Subroutine~\ref{subroutine2} 
in a particular iteration satisfies
\begin{equation}\label{eqn:l7a}
 \sum_{e=l}^{l_{0}^{m}-1}x(e) \leq \sum_{e=l}^{l_{0}^{m}-1}\alpha(e)~~~(l_{t}^{m} \leq l < 
l_{0}^{m})
\end{equation}
with equality for $l=l_{t}^{m}$. Furthermore, it is possible to assign values to $x(e)$ as in step 3 
without violating the feasibility constraints.
\end{lemma}
\vspace*{.15in}

\begin{proof}
 Since $l^{m}_{t} \in I_{m}$, the left inequality in (\ref{l6a}) of Lemma \ref{l6} implies
\begin{equation}
 \sum_{e=l^{m}_{t}}^{l_{0}^{m}-1}H_{e}^{-}(\Gamma_{j}) \leq 
\sum_{e=l^{m}_{t}}^{l_{0}^{m}-1}\alpha(e).
\end{equation}
After rearrangement, we get
\begin{equation}
0 \leq \gamma = \sum_{e=l^{m}_{t}}^{l_{0}^{m}-1}\alpha(e) - 
\sum_{e=l^{m}_{t}}^{l_{0}^{m}-1}H_{e}^{-}(\Gamma_{j})
\leq \sum_{e=l^{m}_{1}}^{l_{0}^{m}-1}\alpha(e) - 
\sum_{e=l^{m}_{1}}^{l_{0}^{m}-1}H_{e}^{-}(\Gamma_{j}).
\end{equation}
because of the choice of $t$ attaining the minimum in (\ref{eqn:mingamma}).
Adding $\sum_{e=l_{1}^{m}}^{l_{0}^{m}-1}H_{e}^{-}(\Gamma_{j})$, we get
\begin{equation} \label{l7c}
 \sum_{e=l_{1}^{m}}^{l_{0}^{m}-1}H_{e}^{-}(\Gamma_{j}) \leq \gamma + 
\sum_{e=l_{1}^{m}}^{l_{0}^{m}-1}H_{e}^{-}(\Gamma_{j}) \leq \sum_{e=l^{m}_{1}}^{l_{0}^{m}-1}\alpha(e) 
\stackrel{(a)}{\leq} \sum_{e=l_{1}^{m}}^{l_{0}^{m}-1}H_{e}^{+}(\Gamma_{j})
\end{equation}
where (a) follows because $l_{1}^{m} \in I_{m}$.
Also, any $l$ satisfying $l_{1}^{m} < l < l_{0}^{m}$ is not in $I_{m}$ and hence by Lemma \ref{l6}
\begin{equation}\label{l7d}
 \sum^{l_{0}^{m}-1}_{e=l}H_{e}^{+}(\Gamma_{j}) \leq \sum^{l_{0}^{m}-1}_{e=l}\alpha(e)~~~(l_{1}^{m} < 
l < l_{0}^{m}).
\end{equation}
From (\ref{l7c}) and  (\ref{l7d}), it is evident that there exists an assignment for $x(e) \in 
[H_e^-(\Gamma_j), H_e^+(\Gamma_j)]$, when $l_{1}^{m} \leq e < l_{0}^{m}$, that gives equality in 
(\ref{assign1}) without violating the feasibility constraints for $l_{1}^{m} \leq e < l_{0}^{m}$, 
i.e., without violating (\ref{eqn:l7a}).

Now, for indices $l$ with $l_{t}^{m} < l < l_{1}^{m}$  and $l \notin I_{m}$ assigning $x(e)$'s  
according to (\ref{assign2}) does not violate feasibility constraints (\ref{eqn:l7a}) because we 
have
\begin{equation}\label{eqn:l7e}
 \sum^{l_{0}^{m}-1}_{e=l}H_{e}^{+}(\Gamma_{j}) < \sum^{l_{0}^{m}-1}_{e=l}\alpha(e).
\end{equation}
Indeed, $x(e)=H_{e}^{-}(\Gamma_{j})$ for $l_{t}^{m} \leq e < l_{1}^{m}$ and $x(e) \leq 
H_{e}^{+}(\Gamma_{j})$
for $l_{1}^{m} \leq e < l_{0}^{m}$ and therefore (\ref{eqn:l7e}) implies
\begin{equation}
\sum^{l_{0}^{m}-1}_{e=l}x(e) \leq  \sum^{l_{0}^{m}-1}_{e=l}H_{e}^{+}(\Gamma_{j}) < 
\sum^{l_{0}^{m}-1}_{e=l}\alpha(e),
\end{equation}
which shows  (\ref{eqn:l7a}).

For indices $l$ with $l_{t}^{m} < l < l_{1}^{m}$ and $l \in I_{m}\cup \{s(j)\}$ assignment of 
$x(e)$'s  according to (\ref{assign2}) is also feasible because
\begin{eqnarray*}
 \sum^{l_{0}^{m}-1}_{e=l}x(e) & =  & \left( \gamma + 
\sum^{l_{0}^{m}-1}_{e=l_{1}^{m}}H_{e}^{-}(\Gamma_{j}) \right) + 
\sum^{l_{1}^{m}-1}_{e=l}H_{e}^{-}(\Gamma_{j}) \\
 & =  & \gamma + \sum^{l_{0}^{m}-1}_{e=l}H_{e}^{-}(\Gamma_{j})  \\
& \stackrel{(b)}{\leq} &    \sum^{l_{0}^{m}-1}_{e=l}\alpha(e) - 
\sum^{l_{0}^{m}-1}_{e=l}H_{e}^{-}(\Gamma_{j}) +\sum^{l_{0}^{m}-1}_{e=l}H_{e}^{-}(\Gamma_{j})  \\
& = & \sum^{l_{0}^{m}-1}_{e=l}\alpha(e)
\end{eqnarray*}
where (b) follows because $\gamma$ is the minimum in (\ref{eqn:mingamma}) among all $l \in I_{m} 
\cup \{s(j)\}$. Moreover, the inequality is an equality when $l=l_{t}^{m}$ because the minimum is 
then attained. This proves that the assignment in step 3 is feasible.
\end{proof}

\vspace*{.15in}
\begin{lemma}\label{l8}
 (Correctness of the algorithm)
If the feasible set is nonempty, Algorithm \ref{algorithm} runs to completion and puts out a 
feasible vector.
\end{lemma}
\vspace*{.15in}

\begin{proof}
Since  $H_{e}^{+}$ and  $H_{e}^{-}$ assume values between $0$ and $\beta(e)$, constraint 
(\ref{eqn:positivityBounded}) is trivially satisfied. It is also straightforward to see that 
constraint (\ref{eqn:ladderConstraint}) is satisfied if and only if
\begin{equation}\label{l8a}
 \sum_{e=l}^{n}x(e) \leq \sum_{e=l}^{n}\alpha(e)~~~(1 < l \leq n)
\end{equation}
hold. Inequalities (\ref{l8a}) are obtained by subtracting (\ref{eqn:ladderConstraint}) from 
(\ref{eqn:equalityConstraint}). Now we show that the vector put out by Algorithm \ref{algorithm} 
satisfies (\ref{l8a}) and (\ref{eqn:equalityConstraint}).

Consider the first iteration of  Algorithm \ref{algorithm}. Since the feasible set is nonempty, from 
Lemma \ref{lemma1}, it follows that there exist an $\eta$ that solves inequality (\ref{eqn:zero}). 
Hence the set over which minimum is taken in (\ref{eqn:gamma}) is nonempty. 
Subroutine~\ref{subroutine1} sets the value of $s(1)$. In the first iteration of 
Subroutine~\ref{subroutine2}, observe that $M_{0}= \{s(1),s(1)+1,\ldots,n\}$ and 
$I_{0}=I=\{l_{1},l_{2},\ldots,l_{p}\}$ a subset of the indices that attain the minimum in 
(\ref{eqn:gamma}), with $p$ the smallest index such that $l_{p}$ is strictly greater than $s(1)$. 
Clearly $I_{0} \cup \{s(1)\}$ and $M_{0} \backslash (I_{0} \cup \{s(1)\})$ satisfy (\ref{l6a}) and 
(\ref{l6b}), respectively. By Lemma \ref{l6} and by induction, in every iteration of 
Subroutine~\ref{subroutine2}, $I_{m} \cup \{s(1)\}$ and $M_{m} \backslash (I_{m} \cup \{s(1)\})$ 
satisfy (\ref{l6a}) and (\ref{l6b}), respectively. Moreover by Lemma \ref{l7}, in every iteration of 
Subroutine~\ref{subroutine2}, we allocate  $x(e)$ for $l_{t}^{m} \leq l < l_{0}^{m}$ such that
\begin{equation}
 \sum_{e=l}^{l_{0}^{m}-1}x(e) \leq \sum_{e=l}^{l_{0}^{m}-1}\alpha(e)
\end{equation}
with equality for $l=l_{t}^{m}$. By induction on $m$, and observing that $l_{0}^{0}=n+1$, the output 
of Subroutine~\ref{subroutine2} satisfies
\begin{equation}\label{l8b}
 \sum_{e=l}^{n}x(e) \leq \sum_{e=l}^{n}\alpha(e)~~~(s(1) \leq l \leq n),
\end{equation}
with equality for $l=s(1)$.

Now assume (\ref{l8b}) is true for $l$ satisfying $s(j-1) \leq l \leq n $. i.e.,
\begin{equation}\label{l8c}
 \sum_{e=l}^{n}x(e) \leq \sum_{e=l}^{n}\alpha(e)~~~(s(j-1) \leq l \leq n),
\end{equation}
with equality for $l=s(j-1)$.
Consider the $j^{th}$ iteration of Algorithm \ref{algorithm}. Observe that 
$M_{0}=\{s(j),s(j)+1,\ldots,s(j-1)-1\}$ and $I_{0}=I=\{l_{1},l_{2},\ldots,l_{p}\}$, a subset of 
indices that attain the minimum in (\ref{eqn:gamma}) with $p$  the smallest index such that $l_{p}$ 
is strictly larger than $s(j)$.
Clearly indices in $I_{0} \cup \{s(j)\}$ and $M_{0} \backslash (I_{0}\cup\{s(j)\})$ satisfy 
(\ref{l6a}) and (\ref{l6b}), respectively. By Lemma  \ref{l6} and by induction, in every iteration 
of Subroutine~\ref{subroutine2}, $I_{m} \cup \{s(j)\}$ and $M_{m} \backslash (I_{m} \cup \{s(j)\})$ 
satisfy (\ref{l6a}) and (\ref{l6b}), respectively.

Moreover, by Lemma \ref{l7}, in every iteration of Subroutine~\ref{subroutine2}, we allocate $x(e)$ 
for $l_{t}^{m} \leq e < l_{0}^{m}$ such that
\begin{equation}
 \sum^{l_{0}^{m}-1}_{e=l} x(e)  \leq  \sum_{e=l}^{l_{0}^{m}-1} \alpha(e)~~~~(l_{t}^{m} \leq l < 
l_{0}^{m} )
\end{equation}
 with equality for $l=l_{t}^{m}$. By induction, the output of Subroutine~\ref{subroutine2} satisfies
\begin{equation} \label{l8d}
 \sum^{s(j-1)-1}_{e=l} x(e)  \leq  \sum_{e=l}^{s(j-1)-1} \alpha(e)~~~~(s(j) \leq l < s(j-1) )
\end{equation}
Combining (\ref{l8c}) and (\ref{l8d}), we see that (\ref{l8d}) holds for $l$ satisfying $s(j) \leq l 
\leq n$ with equality for $l=s(j)$. By induction once again on the $j$ iterations we have
\begin{equation}
  \sum^{n}_{e=l} x(e)  \leq  \sum_{e=l}^{n} \alpha(e)~~~~(1 \leq l \leq n )
\end{equation}
with equality for $l=1$. We have thus verified feasibility.
\end{proof}

\vspace*{.15in}
What remains is the proof of optimality.

\vspace*{.15in}
\begin{lemma}\label{l9}
The vector $x(e)$, $e \in E$, put out by Algorithm \ref{algorithm} is optimal.
\end{lemma}
\vspace*{.15in}

\begin{proof}
We  use Theorem \ref{thm:opt} to prove the optimality of $x$. Let $g = f_{\beta}$. In each 
iteration, at least one variable gets set. So Algorithm \ref{algorithm} terminates after $\tau \leq 
n$ steps. Define $A_{0} = \emptyset$ and set $A_{j} =\{s(j),\cdots,n\} $ for $j=1,2,\cdots,\tau$. 
Observe that $A_{\tau}=E$. It is then an immediate consequence that  $x(A_{j}) = f(A_{j})$. By the 
definition of $\textsf{dep}$, we have  $\textsf{dep}(x,e,f) \subseteq A_{j}$ for every $e \in 
A_{j}-A_{j-1},~ j=1,\cdots,\tau$, and $\textsf{dep}(x,e,g) \subseteq \textsf{dep}(x,e,f)$ for an $x 
\in P(f) \cap P(g)$. Also, observe that  $\textsf{dep}(x,e,g) = \{e\}$ for every $e$ satisfying 
$x(e) = \beta(e)$, and a $u \neq e$ with $x(u) = 0$ cannot belong to  $\textsf{dep}(x,e,g)$.  This 
observation implies that for any $u \in \textsf{dep}(x,e,g)$, $u \neq e$, we must have $x(e) < 
\beta(e)$ and $0 < x(u)$. We then  claim that
\[
 w_{e}^{+}(x(e)) \stackrel{(a)}{\geq} w_{e}^{+}(H_e^{-}(\Gamma_{j})) \stackrel{(b)}{\geq} \Gamma_{j} 
\stackrel{(c)}{\geq} \Gamma_{i}  \stackrel{(d)}{\geq}  w_{u}^{-}(H_{u}^{+}(\Gamma_{i})) 
\stackrel{(e)}{\geq}  w_{u}^{-}(x(u)).
 \]
Inequality (a)  holds since $w_{e}^{+}$ is nondecreasing and $x(e) \geq H_{e}^{-}(\Gamma_{j})$. 
Inequality (b) follows from the definition of $H_{e}^{-}$, after noting that $x(e) < \beta(e)$. 
Since $(e,u)$ is an exchangeable pair and $e \in A_{j}-A_{j-1}$, we must have that $u \in A_i - 
A_{i-1}$ for some $i \leq j$, and  (c) follows from Lemma \ref{l4}. Inequality (d) follows from the 
definition of $w_{u}^{-}$ after noting that $x(u) > 0$. Finally (e) holds  since $w_{u}^{-}$ is 
nondecreasing and $x(u) \leq H_{u}^{+}(\Gamma_{i})$. The sufficient condition of Theorem 
\ref{thm:opt} for optimality holds, and the proof is complete.
\end{proof} 
\vspace*{.15in}

Lemma \ref{l8} and Lemma \ref{l9} imply Theorem \ref{th3}, and its proof is now complete. $\hfill 
\Box$

\end{document}